\documentclass[11pt,a4paper]{amsart}
\usepackage[utf8]{inputenc}

\usepackage{amssymb}
\usepackage{amscd}
\usepackage{amsmath}
\usepackage{mathtools}
\usepackage{eucal}
\usepackage{graphicx}
\usepackage{amsthm}
\usepackage{hyperref}
\usepackage{mathrsfs}
\usepackage{tikz}
\usepackage{float}
\usepackage{geometry}
\usepackage{tikz-cd} 
\usepackage{wrapfig}

\usetikzlibrary{arrows,intersections}
\usetikzlibrary{patterns}
\usepackage{capt-of}

\hypersetup{
	colorlinks=true,
	citecolor=blue,
	linkcolor=blue,
	urlcolor=blue
}

\newtheorem{theorem}{Theorem}[section]
\newtheorem{lemma}[theorem]{Lemma}
\newtheorem{corollary}[theorem]{Corollary}
\newtheorem{proposition}[theorem]{Proposition}
\theoremstyle{definition}

\theoremstyle{remark}

\newtheorem{example}[theorem]{Example}
\newtheorem{conjecture}[theorem]{Conjecture}

\newcommand{\Res}{\operatorname*{Res}}

\title{An Enhanced Higher Green's Function on the Modular Curve}
\author{N.~Sakharova}
\address{National Research University Higher School of Economics, Laboratory of Mirror Symmetry, 6 Usacheva str., Moscow, Russia, 119048}
\email{nsakharova@hse.ru}

\date{}

\begin{document}
	\begin{abstract}
		The higher Green's functions $G_{k/2}(z,w)$ on the modular curve are real-valued eigenfunctions of the hyperbolic Laplacian in each variable, with logarithmic singularities along the diagonal. Gross and Zagier conjectured that their values at CM points, after multiplication by suitable powers of the discriminants, are logarithms of absolute values of algebraic numbers \cite{GZ86}. This conjecture was proved by Li \cite{Li23} and Bruinier, Li, and Yang \cite{BLY25}. In \cite{Mel08}, Mellit proposed a complex refinement of this conjecture. It predicts that the normalized CM values of a complex lift of the Green function are logarithms of algebraic numbers modulo periods.
		
		For $\Gamma=\mathrm{SL}_2(\mathbb Z)$ and $k\in{4,6,8,10,14}$, we follow Zagier's construction of CM invariants from Eichler integrals of the Eisenstein series $E_4$ \cite{GZ00}. By applying Maass raising operators in both variables to a normalized iterated Eichler integral of a Bergman-type kernel, we obtain a complex-valued extension $\mathcal G_{k/2}^\bullet(z,w)$ whose real part is a constant multiple of $G_{k/2}(z,w)$. We study the normalized CM values of $\mathcal G_{k/2}^\bullet$ and describe their dependence on the choice of representatives of quadratic form classes. In particular, we show that replacing the quadratic form corresponding to the first variable by an equivalent one changes the exponential of the normalized value only by a root of unity. We also prove the algebraicity of these exponentials in the special case where the second CM point is $w_*=(1+i\sqrt d)/2$, with $d>3$ and $d\equiv3\pmod4$, and the class of the quadratic form corresponding to the first variable is equal to its inverse.
	\end{abstract}
	
	\maketitle
	
	\section{Introduction}
	\label{sec:introduction}
	
	\subsection{Motivation and main result}
	\label{subsec:background}
	
	Without much loss of generality, we work with the full modular group $\Gamma=\mathrm{SL}_2(\mathbb Z)$. We also restrict to the special case where the space $S_{k}(\Gamma)$ of cusp forms of weight $k$ for $\Gamma$ is trivial: $k \in S:=\{4,6,8,10,14\}$. Let $z=x+iy \in \mathfrak H$. The hyperbolic Laplace--Beltrami operator on the upper half-plane $\mathfrak H$ is given by
	$\Delta_{\mathrm{hyp}} = -y^2\left(
	\frac{\partial^2}{\partial x^2}
	+ \frac{\partial^2}{\partial y^2}
	\right)$. For $z,w\in\mathfrak H$, the resolvent kernel of this operator on $\mathfrak H$ with spectral parameter $s$ is given by
	$$G_s^{\mathfrak H}(z,w) = -2Q_{s-1}\bigl(1+\frac{|z-w|^2}{2\,\Im z\,\Im w}\bigr),$$
	where $Q_{s-1}$ is the Legendre function of the second kind. For $X_\Gamma=\Gamma\backslash \mathfrak H$,
	$\Gamma=\mathrm{PSL}_2(\mathbb Z)$, the higher Green's function on the modular curve \cite{GZ85,GZ86,GKZ87} is defined by the series
	$$G_s(z,w) = \sum_{\gamma\in\Gamma}G_s^{\mathfrak H}(z,\gamma w)~~~ (\Re(s)>1).$$
	Let $k$ be a positive even integer. The functions $G_{k/2}(z,w)$ are $\Gamma$-bi-invariant and real-analytic on
	$(\Gamma\backslash\mathfrak H)^2$ away from the diagonal; they have a logarithmic singularity along the diagonal and are eigenfunctions of the hyperbolic Laplacian in each variable with eigenvalue $s(1-s)$. At the cusp, as $y\to\infty$, the constant term of the Green's function is proportional to the non-holomorphic Eisenstein series $y^{1-s}E_k(w,s)$ \cite{GZ86,Za88}.
	
	For $k \in S$, the Fourier coefficients of $G_{k/2}^{\mathcal{H} / \Gamma}$ are rational numbers, except for the constant term, which involves $\zeta(2k-1)$. Gross and Zagier conjectured that suitably normalized values of higher Green's functions at CM points are logarithms of absolute values of algebraic numbers in the corresponding class fields \cite{GZ86}.
	
	The Gross--Zagier conjecture has been studied for several decades. In \cite{Mel08} Mellit constructed a complex lift of the CM values of higher Green's functions and predicted that their suitably normalized exponentials are algebraic. Using higher Chow cycles, he proved this conjecture for the full modular group in weight $4$, with one of the CM points equal to $i$. In \cite{Via11,Via15}, Viazovska related CM values of higher Green's functions to regularized Petersson inner products and proved the conjecture when both CM points lie in the same imaginary quadratic field. In \cite{Li22}, Li obtained an averaged formula for CM values and a factorization of the corresponding algebraic numbers. Bruinier, Ehlen, and Yang \cite{BEY21} proved a partially averaged version of the conjecture for orthogonal Shimura varieties. The logarithmic algebraicity of individual CM values was subsequently proved by Li \cite{Li23} and by Bruinier, Li, and Yang \cite{BLY25}. More recently, Brown and Fonseca \cite{BF25} gave a geometric and motivic interpretation of the Gross--Zagier conjecture and showed that it follows from a conjecture concerning the action of Hecke operators on the motives of cusp forms.
	
	In the present paper, we use the classical logarithmic algebraicity of higher Green's functions as an established result. To study their complex-valued enhancement at CM points, we consider a two-variable function $H_k(z,w)$ obtained by iterated Eichler integration of the Bergman-type kernel $\Omega_k(z,w)$. Applying Maass raising operators in both variables gives a function $\mathcal G_{k/2}(z,w)$ whose real part agrees with $G_{k/2}(z,w)$ up to a term involving a non-holomorphic Eisenstein series. By adding a suitable polynomial correction to $H_k$, we obtain a function $H_k^\bullet$ such that the real part of the corresponding function after Maass raising in both variables coincides exactly with $G_{k/2}$:
	$2(-1)^{k/2+1}\Re\mathcal G_{k/2}^\bullet=G_{k/2}$.
	
	We then study the CM values of $\mathcal G_{k/2}^\bullet$ and their transformation when a quadratic form is replaced by an equivalent one. For CM points $z$ and $w$, the exponential of the normalized value of $\mathcal G_{k/2}^\bullet(z,w)$ is well defined up to multiplication by a root of unity. As a consequence, for $w_*=(1+i\sqrt d)/2$, $d>3$, $d\equiv3\pmod4$, we obtain a special case of the complex Gross--Zagier conjecture formulated by Mellit \cite{Mel08}: these complex numbers are algebraic if the quadratic-form class associated with the first variable is self-inverse.
	
	The paper is organized as follows. In Section~\ref{sec:preliminaries}, we introduce the functions $\Omega_k$, $F_k$, $H_k$, and $G_{k/2}$, together with the differential relations among them and their Fourier expansions. In Section~\ref{sec:maassgreen}, we establish the relation between the Maass raising of $H_k$ and the higher Green's function.
	In Section~\ref{sec:zagier-normalization}, we compute the Eichler periods of $H_k$, choose a normalization $H_k^\bullet$, and study the dependence of its CM values on the choice of a representative of the quadratic form.
	In Section~\ref{sec:extended-gross-zagier}, we study the arithmetic properties of the CM values of $\mathcal G_{k/2}^\bullet$ and prove a special case of the complex Gross--Zagier conjecture formulated by Mellit.
	
	\subsection*{Acknowledgements}
	The author thanks Don Zagier for a helpful discussion and the Max Planck Institute for Mathematics in Bonn for its hospitality during a visit in 2021. The author is also grateful to Andrey Levin for many valuable conversations.
	
	This work was supported by the HSE University Basic Research Program (HSE-BR-2025-060).
	
	\subsection{Kernels and differential operators}
	\label{sec:preliminaries}
	
	We consider the following four kernels:
	$$
	\Omega_k(z,w),\qquad F_k(z,w),\qquad G_{k/2}(z,w),\qquad H_k(z,w).
	$$
	The first three have biweights $(k,k)$, $(k,2-k)$, and $(0,0)$, respectively. The function $H_k$ is modular of weight $2-k$ in $w$, while in $z$ it is an Eichler integral, transforming with the same weight up to a period polynomial of degree at most $k-2$.
	
	The \emph{Bergman-type kernel} is the holomorphic Poincaré series of weight $k$ in each variable:
	$$
	\refstepcounter{equation}\label{eq:Omega-definition}
	\Omega_k(z,w)
	:=\frac{1}{2}\frac{(k-1)!}{(2 \pi i)^k}\sum_{\gamma\in\Gamma}
	\frac{1}{(-czw-dw+az+b)^{k}},
	\qquad
	\gamma=\begin{pmatrix}a&b\\c&d\end{pmatrix},\eqno(\theequation)
	$$
	which converges absolutely for $k\ge 4$ and is holomorphic away from the diagonal
	$\Gamma$-orbit. Following Asai, Kaneko, and Ninomiya \cite{AKN97}, we introduce the meromorphic kernel of biweight $(k,2-k)$ with a simple pole along the diagonal:
	$$
	\refstepcounter{equation}\label{eq:F-definition}
	F_k(z,w) =\frac{E_k(z)\,E_{14-k}(w)}       {\Delta(w)\bigl(j(w)-j(z)\bigr)}.\eqno(\theequation)
	$$
	Here $E_k$ denotes the normalized Eisenstein series of weight $k$, 
	$$E_k(z) = 1 - \frac{2k}{B_k} \sum_{n \ge 1} \sigma_{k-1}(n) q^n,$$
	where $B_k$ is the $k$-th Bernoulli number, $q=e^{2\pi iz}$, and
	$\sigma_{k-1}(n)=\sum_{d\mid n}d^{k-1}$.
	We also use the standard notation
	$\Delta(z)=\eta^{24}=\frac{E_4(z)^3-E_6(z)^2}{1728}$ and
	$j(z)=1728\frac{E_4(z)^3}{E_4(z)^3-E_6(z)^2}
	=q^{-1}+744+196884q+\ldots$
	for the modular discriminant and the classical modular invariant, respectively.
	
	Finally, for fixed $w$, we define the Eichler integral of $F_k(\,\cdot\,,w)$ in the $z$-variable by
	$$
	\refstepcounter{equation}\label{eq:H-definition-intro}
	H_k(z,w):=\frac{(2\pi i)^{k-1}}{(k-2)!}\int_{i\infty}^{z}(z-t)^{k-2}F_k(t,w)\,dt,\eqno(\theequation)
	$$
	with the following choice of integration path. For fixed $w$, choose
	$Y=\max\{\Im w,(\Im w)^{-1}\}$, so that $\Im(\gamma w)\le Y$ for every $\gamma\in\Gamma$ and hence the orbit $\Gamma w$ lies below the line $\Im t=Y$. For $\Im z>Y$, the integral is taken along the vertical path from $i\infty$ to $z$. The function $H_k(z,w)$ is then extended to $\mathfrak H\setminus\Gamma w$ by analytic continuation along paths avoiding the orbit $\Gamma w$.
	
By construction, $D_z^{k-1}H_k(z,w)=F_k(z,w)$, where $D_z=(2\pi i)^{-1}d/dz$. 
	
	\begin{proposition}
		\label{thm:diagonal-logarithm}
		Let $w\notin\Gamma\cdot\{i,\rho\}$, where $\rho=e^{2\pi i/3}$. Then, for $z$ sufficiently close to $w$ and a fixed local branch of $\log(z-w)$, we have
		$$
		\refstepcounter{equation}\label{eq:H-diagonal-log}
		H_k(z,w)=\frac{(2\pi i)^{k-2}}{(k-2)!}(z-w)^{k-2}\log(z-w)
		+H_{k,\mathrm{reg}}(z,w),\eqno(\theequation)
		$$
		where $H_{k,\mathrm{reg}}$ is holomorphic in a neighborhood of the diagonal.
	\end{proposition}
	
	\begin{proof}
		As $t\to w$, we have $j(w)-j(t)=-j'(w)(t-w)+O((t-w)^2)$.
		Using $\frac{j'(w)}{2\pi i}=-\frac{E_k(w)E_{14-k}(w)}{\Delta(w)}$,
		we obtain $F_k(t,w)=\frac{1}{2\pi i(t-w)}+O(1)$.
		Since $(z-t)^{k-2}=(z-w)^{k-2}+O(t-w)$,
		the singular part of the integrand is	$\frac{(2\pi i)^{k-2}}{(k-2)!}	\frac{(z-w)^{k-2}}{t-w}.$	Integrating this expression gives the logarithmic singularity in 	\eqref{eq:H-diagonal-log}, while the remainder is holomorphic near the diagonal.
	\end{proof}
	Let $\operatorname{Mon}_wH_k$ denote the change in $H_k(z,w)$ under analytic continuation in the $z$-variable along a single positively oriented loop around $z=w$. It follows from \eqref{eq:H-diagonal-log} that
	$$
	\operatorname{Mon}_wH_k(z,w)=\frac{(2\pi i)^{k-1}}{(k-2)!}(z-w)^{k-2}.
	$$
	The four kernels considered here are related by the following differential operators: the holomorphic differential operator $D_z^n=(2\pi i)^{-n}d^n/dz^n$ and the Maass raising operator, which maps functions of weight $\kappa$ to functions of weight $\kappa+2$:
	$$
	\refstepcounter{equation}\label{eq:Maass-raising}
	\partial_\kappa f(z)=Df(z)-\frac{\kappa}{4\pi y}f(z),
	\qquad y=\Im z.\eqno(\theequation)
	$$
	Let
	$$
	\refstepcounter{equation}\label{eq:Maass-composition}
	\mathcal R^{\kappa}_{[n]}
	:=\partial_{\kappa+2(n-1)}\circ\partial_{\kappa+2(n-2)}\circ\cdots\circ\partial_{\kappa+2}\circ\partial_{\kappa},\eqno(\theequation)
	$$
	denote the $n$-fold Maass raising operator, starting in weight $\kappa$ and ending in weight $\kappa+2n$. The four kernels are then related by the following diagram:
	$$
	\refstepcounter{equation}\label{eq:kernel-diagram}
	\begin{tikzcd}[row sep=large,column sep=large]
		H_k(z,w)
		\arrow[r,"D_z^{k-1}"]
		\arrow[dr,"{\mathcal R^{2-k}_{[k/2-1],z}\,\mathcal R^{2-k}_{[k/2-1],w}\,;\,\Re\pmod E}",swap]
		&F_k(z,w)
		\arrow[r,"D_w^{k-1}"]
		&\Omega_k(z,w)
		\\
		&G_{k/2}(z,w)
		\arrow[ur,"{\mathcal R^{0}_{[k/2],z}\,\mathcal R^{0}_{[k/2],w}}",swap]
		&
	\end{tikzcd}
	\eqno(\theequation)
	$$
	where the diagonal arrow is understood up to the addition of the non-holomorphic Eisenstein series $E\left(w,k/2\right)$.
	
	The functions $\Omega_k$, $F_k$, $G_{k/2}$, and $H_k$ admit well-known Fourier expansions in $(q_z,q_w)$, where $q_z=e^{2\pi iz}$ and $q_w=e^{2\pi iw}$, involving modified Bessel functions and Kloosterman sums. Let $p^{(k)}_{-m}(z)$ be the weakly holomorphic Poincaré series of weight $k$ with a pole of order $m$ at infinity:
	$$
	\refstepcounter{equation}
	p^{(k)}_{-m}(z)=\frac{1}{2}\sum_{\gamma \in \Gamma_{\infty} \backslash \Gamma}\frac{e^{-2 \pi i m \gamma z}}{(cz+d)^{k}}.\eqno(\theequation)
	$$
	For the Poincaré series with principal part $q^{-1}$ at infinity, one has
	$$p^{(k)}_{-1}(z)=E_k(z) \left(j(z) - 744 + \frac{2k}{B_k} \right),$$
	whereas the Poincaré series with principal part $q^{-n}$ are obtained from this one by the action of the Hecke operator:
	$p_{-m}^{(k)}=m^{1-k}p_{-1}^{(k)}\big|_kT(m)$.
	Let $m>0$ and write
	$$
	p^{(k)}_{-m}(z)= e^{-2\pi i m z}+\sum_{n>0}\pi^{(k)}(-m, n)e^{2 \pi i n z}.
	$$
	The Fourier coefficients of these Poincaré series are given by
	$$
	\refstepcounter{equation}\label{eq:Poincare}
	\pi^{(k)}(-m, n)=2\pi (-1)^{k/2}\left( \frac{n}{m}\right) ^{(k-1)/2}\sum_{c>0}\frac{K(-m, n; c)}{c}I_{k-1}\left(\frac{4 \pi \sqrt{mn}}{c}\right),\eqno(\theequation)
	$$
	where $K(m,n;c)$ is the standard Kloosterman sum
	$$
	\refstepcounter{equation}\label{eq:Kloosterman-definition}
	K(m,n;c)
	=\!\!\!\sum_{\substack{d\bmod c\\(d,c)=1}}\!\!
	e^{2\pi i(md+n\bar d)/c},
	\qquad d\bar d\equiv 1\pmod c,\eqno(\theequation)
	$$
	with $K(m,n;1)=1$, and $I_{k-1}$ is the modified Bessel function,
	$I_{k-1}(x) = \sum_{n=0}^\infty \frac{(x/2)^{k-1+2n}}{n! \, \Gamma(k + n )}$. \\
	
	The Poincaré series $g_{-m}^{(k)}\in M_{2-k}^!(\Gamma)$ are the unique weakly holomorphic modular forms of weight $2-k$ with principal part $q^{-m}$ at the cusp:
	$$
	\refstepcounter{equation}\label{eq:Hecke-gn}
	g_{-m}^{(k)}=m^{k-1}\,g_{-1}^{(k)}\big|_{2-k}T(m),
	\qquad g_{-1}^{(k)}=E_{14-k}/\Delta,\eqno(\theequation)
	$$
	
	If $g_{-m}^{(k)}(z)=q^{-m}+\sum_{n\ge0}\gamma^{(k)}(-m,n)q^n$, then its constant Fourier coefficient is $\gamma^{(k)}(-m,0)=\frac{2k}{B_k}\sigma_{k-1}(m)$. Since $\frac{1}{(-m)^{k-1}}D^{k-1}g_{-m}^{(k)}(z)=p_{-m}^{(k)}(z)$, comparison of Fourier coefficients gives $\gamma^{(k)}(-m,n)=\left(-\frac{m}{n}\right)^{k-1}\pi^{(k)}(-m,n)$.\\
	
	The kernels $\Omega_k$, $F_k$, and $H_k$ have the following Fourier expansions.
	
	\begin{lemma}\label{lem:fourier-expansions}
		If $\Im z>\max\{\Im w,(\Im w)^{-1}\}$, then:
		\begin{enumerate}
			\item For the kernel $\Omega_k$, we have
			$$
			\refstepcounter{equation}
			\begin{gathered}
				\Omega_k(z,w)
				=\sum_{m>0}m^{k-1}p_{-m}^{(k)}(w)q_z^m
				=\sum_{m>0}m^{k-1}q_z^m q_w^{-m}
				+\sum_{m,n>0}m^{k-1}\pi^{(k)}(-m,n)q_z^m q_w^n.
			\end{gathered}\eqno(\theequation)
			$$
			
			\item For $F_k$, we have
			$$
			\refstepcounter{equation}
			\begin{gathered}
				F_k(z,w)
				=-\sum_{m>0}g_{-m}^{(k)}(w)q_z^m\\
				=-\sum_{m>0}
				\left(q_w^{-m}+\frac{2k}{B_k}\sigma_{k-1}(m)\right)q_z^m
				+\sum_{m,n>0}\left(\frac{m}{n}\right)^{k-1}
				\pi^{(k)}(-m,n)q_z^m q_w^n.
			\end{gathered}\eqno(\theequation)
			$$
			
			\item For $H_k$, we have
			$$
			\refstepcounter{equation}
			\begin{gathered}
				H_k(z,w)
				=-\sum_{m>0}\frac{g_{-m}^{(k)}(w)}{m^{k-1}}q_z^m\\
				=-\operatorname{Li}_{k-1}\left(\frac{q_z}{q_w}\right)
				-\frac{2k}{B_k}\sum_{m>0}
				\frac{\sigma_{k-1}(m)}{m^{k-1}}q_z^m
				+\sum_{m,n>0}
				\frac{\pi^{(k)}(-m,n)}{n^{k-1}}q_z^m q_w^n.
			\end{gathered}\eqno(\theequation)
			$$
		\end{enumerate}
	\end{lemma}
	
	\subsection{Maass raising and the higher Green's function}
	\label{sec:maassgreen}
	
	The relation between $H_k(z,w)$ and the higher Green's function is given by Maass raising operators acting in both variables. Let $K_\nu$ denote the Bessel $K$-function,
	$K_\nu(x)=\frac{\pi}{2}\frac{I_{-\nu}(x)-I_\nu(x)}{\sin(\pi\nu)}$. For half-integral order $\nu$, the function $K_\nu(x)$ is elementary:
	$$
	K_{\frac{k-1}{2}}(x)=\sqrt{\frac{\pi}{2}}\,x^{-\frac{k-1}{2}}e^{-x}h_k(x),\qquad h_k(x)=\sum_{j=0}^{k/2-1}\frac{(k/2-1+j)!}{2^jj!(k/2-1-j)!}\,x^{k/2-1-j}.
	$$
	A direct computation shows that the action of the iterated Maass raising operator $\mathcal R_{[k/2-1]}^{2-k}$ on an exponential is given by
	$\mathcal R_{[k/2-1]}^{2-k}q^r=\frac{h_k(2\pi ry)}{(2\pi y)^{k/2-1}}q^r.$
	For $r>0$, the identity takes the form
	$$
	\mathcal R_{[k/2-1]}^{2-k}q^r=\sqrt{\frac{2}{\pi}}(2\pi y)^{1/2}r^{\frac{k-1}{2}}K_{\frac{k-1}{2}}(2\pi ry)e^{2\pi irx}.
	$$
	
	We next apply $\mathcal R_{[k/2-1]}^{2-k}$ to the Poincaré series $g_{-m}^{(k)}$ of weight $2-k$.
	\begin{lemma}
		$$
		\refstepcounter{equation}
		\begin{gathered}
			\mathcal R^{2-k}_{[k/2-1]}g_{-m}^{(k)}(z)
			=\frac{h_k(-2\pi my)}{(2\pi y)^{k/2-1}}e^{-2\pi imz}
			+\gamma^{(k)}(-m,0)\frac{(k-2)!}{(k/2-1)!(4\pi)^{k/2-1}}y^{1-k/2}\\
			+2y^{1/2}\sum_{n>0}\gamma^{(k)}(-m,n)n^{\frac{k-1}{2}}K_{\frac{k-1}{2}}(2\pi ny)e^{2\pi inx}.
		\end{gathered}\eqno(\theequation)
		$$
	\end{lemma}
	
	In \cite{Ni73} Niebur introduced the real-analytic Poincaré series of weight $0$
	$$
	\mathfrak F_m^{(s)}(z):=\sum_{M\in\Gamma_\infty\backslash\Gamma}e(m\Re Mz)(\Im Mz)^{1/2}I_{s-1/2}(2\pi|m|\Im Mz),
	$$
	which converge absolutely and uniformly for $\Re s>1$ on sets of the form $\{x+iy:|x|\le A,\ y\ge\varepsilon\}$.
	
	We now obtain the Fourier expansion of the function obtained by applying the Maass raising operators to $H_k(z,w)$.
	\begin{theorem}
		Define
		$$
		\mathcal G_{k/2}(z,w)=\mathcal R^{2-k}_{[k/2-1],z}\mathcal R^{2-k}_{[k/2-1],w}H_k(z,w).
		$$
		Let $z=x_z+iy_z$, $w=x_w+iy_w$, and suppose that $y_z>\max\{y_w,y_w^{-1}\}$. Then
		$$
		\refstepcounter{equation}
		\mathcal G_{k/2}(z,w)=-4\pi(-1)^{k/2+1}\sqrt{y_z}\sum_{m>0}\mathfrak F_{-m}^{(k/2)}(w)K_{\frac{k-1}{2}}(2\pi my_z)e^{2\pi imx_z}.\eqno(\theequation)
		$$
	\end{theorem}
	
	The higher Green's function has the Fourier expansion \cite{Za88}
	$$
	\refstepcounter{equation}
	\begin{gathered}
		G_{k/2}(z,w)=\frac{4\pi}{1-k}y_z^{1-k/2}E\left(w,k/2\right)\\
		-4\pi\sqrt{y_z}\sum_{m>0}K_{\frac{k-1}{2}}(2\pi my_z)
		\left(\mathfrak F_m^{(k/2)}(w)e^{-2\pi imx_z}+\mathfrak F_{-m}^{(k/2)}(w)e^{2\pi imx_z}\right),
	\end{gathered}\eqno(\theequation)
	$$
	where $E(z,s)=\sum_{\gamma\in\Gamma_\infty\backslash\Gamma}[y^s]\circ\gamma$ denotes the non-holomorphic Eisenstein series. Comparing this expansion with the preceding theorem and using
	$\overline{\mathfrak F_{-m}^{(k/2)}(w)}=\mathfrak F_m^{(k/2)}(w)$
	gives the following identity.
	
	\begin{theorem}\label{thm:maass-green}
		The real part of the enhanced function $\mathcal G_{k/2}(z,w)$ agrees with the higher Green's function up to a non-holomorphic Eisenstein term:
		$$
		2(-1)^{k/2+1}\Re\mathcal G_{k/2}(z,w)
		=G_{k/2}(z,w)-\frac{4\pi}{1-k}y_z^{1-k/2}E\left(w, k/2\right).
		$$
	\end{theorem}
	
	This identity admits analytic continuation in the $z$-variable. If the orbit $\Gamma w$ does not intersect the positive imaginary axis, the continuation can be carried out along the entire axis.
	
	The Fourier expansion of $H_k$ also yields a representation of the higher Green's function in terms of the single-valued Bloch--Wigner--Ramakrishnan polylogarithms, analogous to Zagier's formula for the higher Green's function on the punctured unit disk \cite[\S3]{Za90}. For odd $\ell\ge1$, let $D_\ell$ denote the single-valued function in Zagier's normalization \cite[\S2]{Za90}:
	$$
	D_\ell(\xi)=
	\Re\sum_{j=0}^{\ell-1}
	\frac{(-\log|\xi|)^j}{j!}\operatorname{Li}_{\ell-j}(\xi)
	+\frac{\log^\ell|\xi|}{2\ell!}.
	$$
	The function $D_\ell$ is single-valued on $\mathbb C\setminus\{0,1\}$ and satisfies
	$D_\ell(\xi^{-1})=D_\ell(\xi)$.
	
	\begin{proposition}
		\label{prop:polylog-expansion}
		\begin{enumerate}
			\item If $y_z>y_w$ and $y_zy_w>1$, then
			$$
			\refstepcounter{equation}\label{zag:H}
			\begin{aligned}
				H_k(z,w)={}&-\operatorname{Li}_{k-1}(q_z/q_w)
				-\frac{2k}{B_k}\sum_{r\ge1}\operatorname{Li}_{k-1}(q_z^r)-\sum_{r,s\ge1}\gamma^{(k)}(-1,rs)\operatorname{Li}_{k-1}(q_z^rq_w^s).
			\end{aligned}
			\eqno(\theequation)
			$$
			\item If $y_zy_w>1$ and $z-w\notin\mathbb Z$, the Green's function admits the representation
			\begin{multline}\label{zag:G}
				G_{k/2}(z,w)=-\frac{4\pi}{k-1}y_z^{1-k/2}E\left(w,\frac{k}{2}\right)\\
				+\frac{1}{(4\pi^2y_zy_w)^{k/2-1}}
				\sum_{\substack{r\ge1,\ s\ge-1\\rs\ge-1}}
				\gamma^{(k)}(-1,rs)\Bigg[
				\frac{(-1)^{k/2}}{2^{k-3}}
				\sum_{\substack{a,b\ge0\\a+b\le k/2-1}}
				\frac{(-1)^{a+b}(k-2-2a)!}{a!(k/2-1-a)!}\\
				\times\frac{(k-2-2b)!}{b!(k/2-1-b)!}
				(2\pi r y_z)^{2a}(2\pi s y_w)^{2b}
				D_{k-1-2a-2b}(q_z^rq_w^s)\\
				-\frac{(2\pi r y_z)^{k-1}+(2\pi s y_w)^{k-1}}{k-1}
				\Bigg],
			\end{multline}
			where $\gamma^{(k)}(-1,-1)=1$. The double series in \eqref{zag:G} converges locally uniformly in the region $y_zy_w>1$.
		\end{enumerate}
	\end{proposition}
	\begin{proof}
		\begin{enumerate}
			\item We first prove the formula for $H_k$. The Selberg identity
			$K(-a,b;c)=\sum_{d\mid(a,b,c)}dK(-ab/d^2,1;c/d)$
			implies the following relation between the coefficients of the Poincaré series:
			$$
			\frac{\pi^{(k)}(-a,b)}{b^{k-1}}
			=-\sum_{d\mid(a,b)}\frac{\gamma^{(k)}(-1,ab/d^2)}{d^{k-1}}.
			$$
			Substituting this identity into the double Fourier series for $H_k$ and writing $a=dr$, $b=ds$, we obtain the sum
			$\sum_{d\ge1}(q_z^rq_w^s)^d/d^{k-1}
			=\operatorname{Li}_{k-1}(q_z^rq_w^s)$.
			Similarly,
			$\sum_{n\ge1}\sigma_{k-1}(n)q_z^n/n^{k-1}
			=\sum_{r\ge1}\operatorname{Li}_{k-1}(q_z^r)$.
			This proves the first assertion.
			
			\item We now apply the Maass raising operators in both variables to \eqref{zag:H}. For $\xi=q_z^rq_w^s$, we have
			$D_z\operatorname{Li}_n(\xi)=r\operatorname{Li}_{n-1}(\xi)$ and
			$D_w\operatorname{Li}_n(\xi)=s\operatorname{Li}_{n-1}(\xi)$.
			The resulting expression is a linear combination of the real parts of the polylogarithms $\operatorname{Li}_j(\xi)$, with coefficients depending on $y_z$ and $y_w$. Since
			$\log|\xi|=-2\pi(ry_z+sy_w)$, the definition of $D_\ell$ allows us to write this combination as
			$$
			\begin{aligned}
				&2(-1)^{k/2+1}
				\Re\left(
				\mathcal R^{2-k}_{[k/2-1],z}
				\mathcal R^{2-k}_{[k/2-1],w}
				\operatorname{Li}_{k-1}(\xi)
				\right)\\
				&=\frac{1}{(4\pi^2y_zy_w)^{k/2-1}}
				\Bigg[
				\frac{(-1)^{k/2}}{2^{k-3}}
				\sum_{\substack{a,b\ge0\\a+b\le k/2-1}}
				\frac{(-1)^{a+b}(k-2-2a)!(k-2-2b)!}
				{a!b!(k/2-1-a)!(k/2-1-b)!}\\
				&\hspace{30mm}\times(2\pi r y_z)^{2a}(2\pi s y_w)^{2b}
				D_{k-1-2a-2b}(\xi)
				-\frac{(2\pi r y_z)^{k-1}+(2\pi s y_w)^{k-1}}{k-1}
				\Bigg].
			\end{aligned}
			$$
			Substituting this expression into \eqref{zag:H} and applying Theorem~\ref{thm:maass-green}, we obtain \eqref{zag:G} for $y_z>y_w$ and $y_zy_w>1$.
			
			It remains to extend \eqref{zag:G} to the full region $y_zy_w>1$. For $r,s\ge1$, the coefficients satisfy the standard estimate $|\gamma^{(k)}(-1,rs)|\le Ce^{4\pi\sqrt{rs}}$, while on every compact subset of this region there exists $\delta>0$ such that
			$ry_z+sy_w-2\sqrt{rs}\ge\delta(r+s)$.
			The polynomial terms in $D_\ell(\xi)$ combine to
			$\bigl((2\pi r y_z)^{k-1}+(2\pi s y_w)^{k-1}\bigr)/(k-1)$
			and therefore cancel with the last term in the square brackets in \eqref{zag:G}. The remaining part decays exponentially, and hence the double series converges uniformly on compact subsets. Finally, since $D_\ell$ is single-valued, the term corresponding to $(r,s)=(1,-1)$ extends to the region $z-w\notin\mathbb Z$. This completes the proof of the second assertion.
		\end{enumerate}
	\end{proof}
	\section{Eichler periods and the normalized CM invariant}
	\label{sec:zagier-normalization}
	
	The aim of this section is to construct a complex-valued CM invariant from the Eichler integral $H_k(z,w)$, following Zagier's construction \cite[\S7]{GZ00}. Unlike the classical Green function, $H_k$ is modular in the variable $z$ only up to a period polynomial of degree at most $k-2$. Consequently, its value at a CM point $z_Q$ depends on the choice of representative $Q$ of the corresponding quadratic-form class, and this dependence persists after applying Maass raising operators.
	
	We first compute the Eichler periods $\Pi_j^{(k)}(w)$ and their images under Maass raising operators. Using these periods, we choose a normalization $H_k^\bullet$ for which the corresponding complex-valued extension $\mathcal G_{k/2}^\bullet$ has real part $\frac12(-1)^{k/2+1}G_{k/2}$. We show that the normalized CM value $D_Q^{(k-2)/4}\mathcal G_{k/2}^\bullet(z_Q,w)$ depends only on the class of $Q$ modulo an explicitly determined period module. If the second variable is also a CM point $w_R$, multiplication by $D_R^{(k-2)/4}$ maps this module into a discrete subgroup of $2\pi i\mathbb Q$.

	\subsection{The period cocycle}
	\label{subsec:period-cocycle}
	
	Fix $w\in\mathfrak H$ and assume that the orbit $\Gamma w$ does not intersect the positive imaginary axis. Set $h_w^{(k)}(z):=H_k(z,w)$ and $C_k:=(2\pi i)^{k-1}/(k-2)!$. Define the Eichler periods by
	$$
	\refstepcounter{equation}\label{eq:Pi-periods}
	\Pi_j^{(k)}(w):=\int_0^{i\infty}u^jF_k(u,w)du,\quad0\le j\le k-2.\eqno(\theequation)
	$$
	The integral is taken along the positive imaginary axis. Let $S\mapsto-1/z$ and $T\mapsto z+1$. We first derive the transformation formulas for $S$ with $z\in i\mathbb R_{>0}$ and for $T$ with $\Im z$ sufficiently large.
	
	Since $F_k$ is $T$-invariant, we have
	$$
	\refstepcounter{equation}\label{eq:T-transformation-hw}
	h_w^{(k)}(z+1)=h_w^{(k)}(z).\eqno(\theequation)
	$$
	The substitution $u=-1/v$ gives
	$$
	z^{k-2}h_w^{(k)}(-1/z)-h_w^{(k)}(z)=C_k\int_0^{i\infty}(z-u)^{k-2}F_k(u,w)du.
	$$
	Expanding $(z-u)^{k-2}$, we obtain
	$$
	\refstepcounter{equation}\label{eq:S-transformation-hw}
	z^{k-2}h_w^{(k)}(-1/z)=h_w^{(k)}(z)+C_k\sum_{j=0}^{k-2}(-1)^j\binom{k-2}{j}z^{k-2-j}\Pi_j^{(k)}(w).\eqno(\theequation)
	$$
	Thus, the coefficients of the polynomial appearing in the $S$-transformation formula are expressed in terms of the periods $\Pi_j^{(k)}(w)$. The identity $F_k(u,w+1)=F_k(u,w)$ gives
	$$
	\refstepcounter{equation}\label{eq:T-invariance-Pi}
	\Pi_j^{(k)}(w+1)=\Pi_j^{(k)}(w),\eqno(\theequation)
	$$
	while the substitution $u=-1/v$ gives
	$$
	\refstepcounter{equation}\label{eq:Pi-reflection}
	\Pi_{k-2-j}^{(k)}(w)=(-1)^{j+1}\Pi_j^{(k)}(w),\quad0\le j\le k-2.\eqno(\theequation)
	$$
	In particular, when $k\equiv2\pmod4$, the middle period vanishes.
	
	Let $\gamma=\begin{psmallmatrix}a&b\\c&d\end{psmallmatrix}\in\Gamma$. Choose a word in the generators $S$ and $T$ representing $\gamma$, and analytically continue $h_w^{(k)}$ along the corresponding path. The transformation formulas for $S$ and $T$ imply that
	$(cz+d)^{k-2}h_w^{(k)}(\gamma z)-h_w^{(k)}(z)$
	is a polynomial in $z$ of degree at most $k-2$, with coefficients in
	$$
	C_k\operatorname{span}_{\mathbb Q}\{\Pi_0^{(k)}(w),\ldots,\Pi_{k-2}^{(k)}(w)\}.
	$$
	
	\subsection{Computation of the periods and the Maass-raised periods}
	\label{subsec:zagier-periods}
	\label{subsection:raised-eichler-periods}
	
	In this subsection we find the periods $\Pi_j^{(k)}(w)$ and their images under Maass raising operators. Similar periods of meromorphic modular forms were studied by Löbrich and Schwagenscheidt \cite{LoebrichSchwagenscheidt} in connection with locally harmonic Maass forms.
	
	To compute these periods, we apply the residue theorem to a closed contour obtained from the path $[0,i\infty]$ and its two modular images. Let $\Delta_\sigma$ be the open ideal triangle with vertices $i\infty,0,-1$, where $\sigma:=ST=\begin{psmallmatrix}0&-1\\1&1\end{psmallmatrix}$. Its sides lie on the lines $\Re t=0$, $\Re t=-1$, and the circle $|t+1/2|=1/2$. We orient the boundary clockwise, along the sides $[i\infty,0]$, $[0,-1]$, $[-1,i\infty]$.
	
	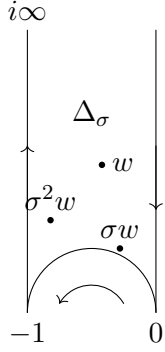
\begin{wrapfigure}{l}{0.35\textwidth}
		\centering
		\begin{tikzpicture}[scale=1.7]
			\draw (-1,0)--(-1,2.2);
			\draw (0,0)--(0,2.2);
			\draw (-1,0) arc (180:0:0.5);
			
			\draw[->] (-1,0.8)--(-1,1.3);
			\draw[->] (0,1.3)--(0,0.8);
			\draw[->] (-0.25,0.07) arc (30:150:0.29);
			
			\fill (-0.42,1.15) circle (0.025) node[right] {$w$};
			\fill (-0.28,0.5) circle (0.025) node[above] {$\sigma w$};
			\fill (-0.82,0.72) circle (0.025) node[above] {$\sigma^2w$};
			
			\node at (-0.5,1.55) {$\Delta_\sigma$};
			\node[below] at (-1,0) {$-1$};
			\node[below] at (0,0) {$0$};
			\node[above] at (-1,2.2) {$i\infty$};
		\end{tikzpicture}
		\caption{The triangle $\Delta_\sigma$.}
		\label{fig:delta-sigma}
	\end{wrapfigure}
	
	The transformation $\sigma(t)=-1/(t+1)$ cyclically permutes the sides of $\Delta_\sigma$. Hence each boundary integral can be expressed in terms of the periods $\Pi_j^{(k)}(w)$, while the integral over the closed contour can be evaluated by the residue theorem.
	
	For this calculation, we may choose $w\in\Delta_\sigma$, provided that $\Gamma w$ does not meet the positive imaginary axis. Indeed, let $v\in\Gamma w$ lie in the standard fundamental domain $|\Re v|\le1/2$, $|v|\ge1$. Then $\Re v\ne0$, and either $v$ or $v-1$ lies in $\Delta_\sigma$. Thus, for a CM point satisfying this assumption, we fix a representative $R$ of the corresponding quadratic-form class such that $w_R\in\Delta_\sigma$.
	
	\begin{proposition}
		\label{thm:Pi-closed-form-general}
		Let $w\in\Delta_\sigma$ and $w\notin\Gamma\cdot{i,\rho}$. Then
		$$
		\refstepcounter{equation}\label{eq:contour-identity-general}
		\begin{gathered}
			\sum_{j=0}^{k-2}\binom{k-2}{j}\Pi_j^{(k)}(w)
			\left((-1)^jz^{k-2-j}+(z+1)^{k-2-j}z^j-(-1)^j(z+1)^{k-2-j}\right)\\
			=(z-w)^{k-2}+\bigl((w+1)z+1\bigr)^{k-2}+\bigl(wz+w+1\bigr)^{k-2}.
		\end{gathered}\eqno(\theequation)
		$$
	\end{proposition}
	
	\begin{proof}
		We evaluate the integral around the boundary of $\Delta_\sigma$ in two ways.
		The identities $j'(t)/(2\pi i)=-E_{14}(t)/\Delta(t)$ and
		$E_{14}=E_kE_{14-k}$ imply that, for
		$\gamma=\begin{psmallmatrix}a&b\\c&d\end{psmallmatrix}\in\Gamma$,
		$\Res_{t=\gamma w}F_k(t,w)=\frac{(cw+d)^{k-2}}{2\pi i}.$
		The three points of the orbit $\Gamma w$ lying inside $\Delta_\sigma$ are
		$w,\sigma w,\sigma^2w$.  
		Since the boundary is oriented clockwise, the residue theorem gives
		$$
		\begin{aligned}
			\oint_{\partial\Delta_\sigma}(z-t)^{k-2}F_k(t,w)dt
			={}&-(z-w)^{k-2}\\
			&-\bigl((w+1)z+1\bigr)^{k-2}
			-\bigl(wz+w+1\bigr)^{k-2}.
		\end{aligned}
		$$
		
		On the second and third sides, make the substitutions
		$t=-1/(s+1)$ and $t=s-1$, respectively. Modularity and periodicity of $F_k$ give
		$$
		\begin{aligned}
			\oint_{\partial\Delta_\sigma}(z-t)^{k-2}F_k(t,w)dt
			={}&\int_0^{i\infty}\Bigl(-(z-s)^{k-2}\\
			&-(zs+z+1)^{k-2}+(z+1-s)^{k-2}\Bigr)F_k(s,w)ds.
		\end{aligned}
		$$
		Expanding the powers and comparing the two expressions for the contour integral,
		we obtain \eqref{eq:contour-identity-general}.
	\end{proof}
	
	\begin{example}
		\label{ex:Pi-closed-forms}
		Proposition~\ref{thm:Pi-closed-form-general} gives
		$\Pi_1^{(4)}(w)=w^2+w+1$ and
		$\Pi_1^{(6)}(w)=(w^2+w+1)^2/4$.
		The formulas for $k=8,10,14$ are given in Appendix~\ref{appendix:explicit-periods}.
	\end{example}
	
	Since the $q$-expansions of $E_k$, $\Delta$, and $j$ have real coefficients,
	$\overline{F_k(z,w)}=F_k(-\bar z,-\bar w)$. Therefore,
	$$
	\refstepcounter{equation}\label{eq:Pi-conjugation}
	\overline{\Pi_j^{(k)}(w)}=(-1)^{j+1}\Pi_j^{(k)}(-\bar w).\eqno(\theequation)
	$$

	To pass from $H_k$ to $\mathcal G_{k/2}$, we apply the Maass raising operators to these periods. Define
	$$
	\refstepcounter{equation}\label{definition:raised-eichler-periods}
	\widehat\Pi_j^{(k)}(w):=\mathcal R^{2-k}_{[k/2-1],w}\Pi_j^{(k)}(w),
	\quad0\le j\le k-2.
	\eqno(\theequation)
	$$
	The reflection relation \eqref{eq:Pi-reflection} remains valid after raising:
	$$
	\refstepcounter{equation}\label{relation:raised-periods-reflection}
	\widehat\Pi_{k-2-j}^{(k)}(w)=(-1)^{j+1}\widehat\Pi_j^{(k)}(w),\quad0\le j\le k-2.
	\eqno(\theequation)
	$$
	
	\begin{proposition}
		\label{theorem:raised-periods-real-structure}
		Let $k\in S$, $d\in\mathbb Z$, $d>3$, and $w_*=(1+i\sqrt d)/2$. Then
		$$
		\refstepcounter{equation}\label{relation:raised-periods-conjugation}
		\overline{\widehat\Pi_j^{(k)}(w_*)}
		=(-1)^{j+1}\widehat\Pi_j^{(k)}(w_*),
		\quad0\le j\le k-2.
		\eqno(\theequation)
		$$
		Moreover, for $1\le j\le k-3$,
		$$
		\refstepcounter{equation}\label{relation:raised-periods-algebraicity}
		\widehat\Pi_j^{(k)}(w_*)\in
		\begin{cases}
			\pi^{-(k-2)/2}\mathbb Q(\sqrt d),&j\text{ odd},\\
			i\pi^{-(k-2)/2}\mathbb Q(\sqrt d),&j\text{ even}.
		\end{cases}
		\eqno(\theequation)
		$$
	\end{proposition}
	
	\begin{proof}
		Equation \eqref{eq:Pi-conjugation} and the definition of the raised periods give
		$\overline{\widehat\Pi_j^{(k)}(w)}=(-1)^{j+1}\widehat\Pi_j^{(k)}(-\bar w).$
		Since $-\bar w_*=w_*-1$ and the raised periods are $T$-invariant, we obtain
		\eqref{relation:raised-periods-conjugation}.
		
		For $1\le j\le\lfloor(k-2)/2\rfloor$, the period $\Pi_j^{(k)}(w)$ is a polynomial in $w$ with rational coefficients on $\Delta_\sigma$. After applying the raising operators and substituting
		$w_*-1=-\frac12+\frac{i\sqrt d}{2}$, we obtain
		$$
		\widehat\Pi_j^{(k)}(w_*)
		\in\pi^{-(k-2)/2}\mathbb Q(i,\sqrt d).
		$$
		By \eqref{relation:raised-periods-conjugation}, this value is real for odd $j$ and purely imaginary for even $j$. Hence \eqref{relation:raised-periods-algebraicity} holds for $1\le j\le\lfloor(k-2)/2\rfloor$. For the remaining indices, the result follows from
		$\widehat\Pi_{k-2-j}^{(k)}(w)=(-1)^{j+1}\widehat\Pi_j^{(k)}(w)$.
	\end{proof}
	
	In particular, for $k=4$,
	$\widehat\Pi_1^{(4)}(w_*)=\frac{d+3}{4\pi\sqrt d}.$
	Explicit formulas for the other weights are given in Appendix~\ref{appendix:explicit-raised-periods}.
	
	The contour identity \eqref{eq:contour-identity-general} does not determine the period $\Pi_0^{(k)}$. For our purposes, it is enough to compute the imaginary part of the raised period $\widehat\Pi_0^{(k)}$. We obtain it by comparing the $S$-transformation formula for $\mathcal G_{k/2}$ with the relation between $\Re\mathcal G_{k/2}$ and the higher Green function $G_{k/2}$.
	
	\begin{theorem}
		\label{theorem:raised-Pi0-Eisenstein}
		Let $k\in S$ and $w\in\Delta_\sigma\setminus\Gamma\cdot{i,\rho}$. Then
		$$
		\refstepcounter{equation}\label{formula:raised-Pi0-Eisenstein}
		\Im\widehat\Pi_0^{(k)}(w)=-\frac{\left(\frac{k}{2}-1\right)!}{(k-1)\pi^{(k-2)/2}}E\left(w,\frac{k}{2}\right).\eqno(\theequation)
		$$
	\end{theorem}
	
	\begin{proof}
		Apply the operators $\mathcal R^{2-k}_{[k/2-1],z}\mathcal R^{2-k}_{[k/2-1],w}$ to \eqref{eq:S-transformation-hw}. Using
		$$
		\mathcal R^{2-k}_{[k/2-1]}\left(z^{k-2}f(-1/z)\right)=\left(\mathcal R^{2-k}_{[k/2-1]}f\right)(-1/z),
		$$
		we obtain
		$$
		\mathcal G_{k/2}(-1/z,w)-\mathcal G_{k/2}(z,w)
		=C_k\sum_{j=0}^{k-2}(-1)^j\binom{k-2}{j}\mathcal R^{2-k}_{[k/2-1],z}z^{k-2-j}\widehat\Pi_j^{(k)}(w).
		$$
		For $z=iy$,
		$$
		\mathcal R^{2-k}_{[k/2-1]}z^{k-2}
		=\frac{(k-2)!}{(\frac{k}{2}-1)!(4\pi)^{k/2-1}}y^{k/2-1},
		$$
		while for $j\ge1$,
		$$
		\mathcal R^{2-k}_{[k/2-1]}z^{k-2-j}=O(y^{k/2-1-j}).
		$$
		Hence, as $y\to\infty$,
		$$
		\Re\bigl(\mathcal G_{k/2}(i/y,w)-\mathcal G_{k/2}(iy,w)\bigr)
		=\frac{2(-1)^{k/2}\pi^{k/2}}{(\frac{k}{2}-1)!}y^{k/2-1}\Im\widehat\Pi_0^{(k)}(w)+O(y^{k/2-2}).
		$$
		On the other hand, Theorem~\ref{thm:maass-green} and the modular invariance of the higher Green function $G_{k/2}$ give
		$$
		\Re\bigl(\mathcal G_{k/2}(i/y,w)-\mathcal G_{k/2}(iy,w)\bigr)
		=\frac{2(-1)^{k/2+1}\pi}{k-1}\left(y^{k/2-1}-y^{1-k/2}\right)E\left(w,\frac{k}{2}\right).
		$$
		Comparing the coefficients of $y^{k/2-1}$ in these two asymptotic formulas gives \eqref{formula:raised-Pi0-Eisenstein}.
	\end{proof}
	
	The values of $E(w,s)$ at points of complex multiplication can be expressed in terms of partial zeta functions. Let $d>3$ be squarefree with $d\equiv3\pmod4$, and set $K=\mathbb Q(\sqrt{-d})$ and $w_*=(1+i\sqrt d)/2$. The point $w_*$ corresponds to the principal ideal class of $K$, and $w_*-1\in\Delta_\sigma$. Since $E(w,s)$ and the raised periods are $T$-invariant, we can apply Theorem~\ref{theorem:raised-Pi0-Eisenstein} at $w_*-1$ to obtain a formula for $\Im\widehat\Pi_0^{(k)}(w_*)$. Using
	$$
	E(w_*,s)=\frac{(d/4)^{s/2}}{\zeta(2s)}\zeta_{K,[1]}(s),
	$$
	where $\zeta_{K,[1]}$ denotes the partial zeta function of the principal ideal class, and Proposition~\ref{theorem:raised-periods-real-structure}, which shows that $\widehat\Pi_0^{(k)}(w_*)$ is purely imaginary, we obtain the following corollary.
	
	\begin{corollary}
		\label{corollary:raised-Pi0-CM}
		Let $k\in S$ and let $d>3$ be squarefree with $d\equiv3\pmod4$. Set
		$K=\mathbb Q(\sqrt{-d})$ and $w_*=(1+i\sqrt d)/2$. Then
		$$
		\refstepcounter{equation}\label{formula:raised-Pi0-CM}
		\widehat\Pi_0^{(k)}(w_*)=-\frac{i\left(\frac{k}{2}-1\right)!}{(k-1)\pi^{(k-2)/2}}
		\frac{(d/4)^{k/4}}{\zeta(k)}
		\zeta_{K,[1]}\left(\frac{k}{2}\right).
		\eqno(\theequation)
		$$
	\end{corollary}
	
	The full formula for the holomorphic period $\Pi_0^{(k)}$ for every $k\in S$ is given in Appendix~\ref{appendix:Pi0}.
	
	\subsection{Normalization of $H_k$}
	\label{subsec:zagier-normalized-monodromy}
	
	In the original normalization of $H_k$, the period $\Pi_0^{(k)}(w)$ appears in the $S$-transformation formula. After applying Maass raising operators, this period gives rise to the non-holomorphic Eisenstein term in $\Re\mathcal G_{k/2}$. Following Zagier \cite[\S7]{GZ00}, we add to $H_k$ a polynomial in $z$ of degree at most $k-2$, chosen so that the resulting function transforms with weight $2-k$ under $S$, while its $T$-transformation involves only the periods $\Pi_j^{(k)}(w)$ with $j\ge1$. Define
	$$
	\refstepcounter{equation}\label{eq:H-bullet-definition}
	H_k^\bullet(z,w):=\frac12\left(H_k(z,w)+z^{k-2}H_k(-1/z,w)\right)-\frac{C_k}{2}(z^{k-2}+1)\Pi_0^{(k)}(w).\eqno(\theequation)
	$$
	Using \eqref{eq:S-transformation-hw}, we can equivalently write
	$$
	\refstepcounter{equation}\label{eq:H-bullet-expanded}
	H_k^\bullet(z,w)=H_k(z,w)-C_k\Pi_0^{(k)}(w)+\frac{C_k}{2}\sum_{j=1}^{k-3}(-1)^j\binom{k-2}{j}z^{k-2-j}\Pi_j^{(k)}(w).\eqno(\theequation)
	$$
	The transformation law of $H_k^\bullet$ under $S$ follows directly from its symmetric definition \eqref{eq:H-bullet-definition}, while the transformation law under $T$ follows from \eqref{eq:H-bullet-expanded} and the $T$-invariance of $H_k$.
	
	\begin{proposition}
		\label{prop:H-bullet-transformations}
		For $k\in S$, we have
		$$
		\refstepcounter{equation}\label{eq:H-bullet-S}
		z^{k-2}H_k^\bullet(-1/z,w)=H_k^\bullet(z,w),\eqno(\theequation)
		$$
		and
		$$
		\refstepcounter{equation}\label{eq:H-bullet-T}
		\begin{gathered}
			H_k^\bullet(z+1,w)-H_k^\bullet(z,w)=\frac{C_k}{2}\sum_{j=1}^{k-3}(-1)^j\binom{k-2}{j}\left((z+1)^{k-2-j}-z^{k-2-j}\right)\Pi_j^{(k)}(w).
		\end{gathered}\eqno(\theequation)
		$$
	\end{proposition}
	
	Finally, define the normalized complex-valued extension of the higher Green function by
	$$
	\refstepcounter{equation}\label{eq:G-bullet-definition}
	\mathcal G_{k/2}^\bullet(z,w):=\mathcal R^{2-k}_{[k/2-1],z}\mathcal R^{2-k}_{[k/2-1],w}H_k^\bullet(z,w).\eqno(\theequation)
	$$
	Applying the Maass raising operators in both variables to \eqref{eq:H-bullet-definition} and taking real parts, we obtain the following corollary from Theorems~\ref{thm:maass-green} and \ref{theorem:raised-Pi0-Eisenstein} and the modular invariance of $G_{k/2}$.
	
	\begin{corollary}
		\label{thm:G-bullet-real-part}
		Let $k\in S$ and $w\in\Delta_\sigma\setminus\Gamma\cdot\{i,\rho\}$. Then
		$$
		\refstepcounter{equation}\label{eq:G-bullet-exact-real-part}
		2(-1)^{k/2+1}\Re\mathcal G_{k/2}^\bullet(z,w)=G_{k/2}(z,w).
		\eqno(\theequation)
		$$
	\end{corollary}
	
	
	\subsection{Normalized CM values of the extended Green function}
	\label{subsec:exact-CM-period-module}
	
We now consider the values of $\mathcal G_{k/2}^\bullet$ at CM points. Although $\Re\mathcal G_{k/2}^\bullet$ is modular invariant, the complex values may depend on the choice of representative of a quadratic-form class.
	
	Let $Q(X,Y)=aX^2-bXY+cY^2$ be a positive-definite integral binary quadratic form of discriminant $D_Q=b^2-4ac<0$, and let
	$$
	z_Q=\frac{b+\sqrt{D_Q}}{2a},\qquad\sqrt{D_Q}=i\sqrt{|D_Q|}.
	$$
	If $z_Q=x_Q+iy_Q$, then
	$$
	\refstepcounter{equation}\label{eq:CM-homogeneous-coordinates}
	\frac{1}{2iy_Q}=\frac{a}{\sqrt{D_Q}},\quad
	\frac{x_Q}{iy_Q}=\frac{b}{\sqrt{D_Q}},\quad
	\frac{|z_Q|^2}{2iy_Q}=\frac{c}{\sqrt{D_Q}}.
	\eqno(\theequation)
	$$
	
	Applying the Maass raising operators to the period polynomial in the transformation law for $H_k$, we obtain a homogeneous polynomial of degree $k/2-1$ in $1/(2iy)$, $x/(iy)$, and $|z|^2/(2iy)$. At $z=z_Q$, these quantities equal $a/\sqrt{D_Q}$, $b/\sqrt{D_Q}$, and $c/\sqrt{D_Q}$, respectively. Thus, multiplication by $(\sqrt{D_Q})^{k/2-1}$ gives a homogeneous polynomial in $a,b,c$ with coefficients expressed in terms of the Eichler periods. Following Zagier \cite[(23)]{GZ00}, we define
	$$
	\refstepcounter{equation}\label{eq:I-bullet-definition}
	\widehat I_{k/2}^\bullet(Q;w):=D_Q^{(k-2)/4}\mathcal G_{k/2}^\bullet(z_Q,w).
	\eqno(\theequation)
	$$
	It follows from \eqref{eq:H-bullet-S} that $\widehat I_{k/2}^\bullet(Q;w)$ is invariant under $S$:
	$$
	\widehat I_{k/2}^\bullet([c,-b,a];w)=\widehat I_{k/2}^\bullet([a,b,c];w).
	$$
	Set
	$$
	P_Q(u):=Q(u-1,1)^{k/2-1}-Q(u,1)^{k/2-1}-Q(-1,1)^{k/2-1}+Q(0,1)^{k/2-1}.
	$$
	
	\begin{proposition}
		\label{prop:I-bullet-T}
		Let $k\in S$ and $z_Q\notin\Gamma w$. Under the transformation
		$T:[a,b,c]\mapsto[a,b+2a,a+b+c]$, one has
		\begin{multline}
			\label{eq:I-bullet-T-general}
			\widehat I_{k/2}^\bullet([a,b+2a,a+b+c];w)-\widehat I_{k/2}^\bullet([a,b,c];w)\\
			=\frac{(-1)^{k/2-1}(2\pi i)^{k/2}}{2(\frac{k}{2}-1)!}
			\mathcal R^{2-k}_{[k/2-1],w}\int_0^{i\infty}P_Q(u)F_k(u,w)\,du.
		\end{multline}
		Since $P_Q(0)=0$, the right-hand side is a linear combination of the periods
		$\widehat\Pi_j^{(k)}(w)$ with $j\ge1$.
	\end{proposition}
	
	\begin{proof}
		We use the identity
		$$
		\refstepcounter{equation}\label{eq:raised-CM-kernel-bullet}
		D_Q^{(k-2)/4}\left.\mathcal R^{2-k}_{[k/2-1],z}(z-u)^{k-2}\right|_{z=z_Q}
		=\frac{(-1)^{k/2-1}(k-2)!}{\left(\frac{k}{2}-1\right)!(2\pi i)^{k/2-1}}Q(u,1)^{k/2-1}.
		\eqno(\theequation)
		$$
		Indeed,
		$$
		\mathcal R^{2-k}_{[k/2-1],z}(z-u)^{k-2}
		=\frac{(k-2)!}{\left(\frac{k}{2}-1\right)!(2\pi i)^{k/2-1}}
		\left((z-u)\left(1-\frac{z-u}{2i\Im z}\right)\right)^{k/2-1},
		$$
		and at $z=z_Q$,
		$$
		(z_Q-u)\left(1-\frac{z_Q-u}{2i\Im z_Q}\right)
		=-\frac{Q(u,1)}{\sqrt{D_Q}}.
		$$
		By \eqref{eq:H-bullet-definition} and the $T$-invariance of $H_k$,
		\begin{multline}
			H_k^\bullet(z+1,w)-H_k^\bullet(z,w)\\
			=\frac{C_k}{2}\int_0^{i\infty}
			\left((z+1-u)^{k-2}-(z-u)^{k-2}-(z+1)^{k-2}+z^{k-2}\right)F_k(u,w)\,du.
		\end{multline}
		Applying the Maass raising operators in both variables, multiplying by $D_Q^{(k-2)/4}$, and setting $z=z_Q$, we obtain the polynomial $P_Q(u)$ in the integrand by \eqref{eq:raised-CM-kernel-bullet}. Since $T$ preserves the discriminant, the left-hand side gives precisely the difference stated in the proposition.
	\end{proof}

	Thus $\widehat I_{k/2}^\bullet(Q;w)$ is invariant under $S$, while its change under $T$ is given by \eqref{eq:I-bullet-T-general}. For fixed $w$, let $\Lambda_{k,w}\subset\mathbb C$ be the subgroup generated by all differences appearing in this formula as $Q$ ranges over the positive-definite integral quadratic forms. Since $S$ and $T$ generate $\Gamma$, equivalent forms $Q$ and $Q'$ satisfy
	$$
	\widehat I_{k/2}^\bullet(Q';w)-\widehat I_{k/2}^\bullet(Q;w)\in\Lambda_{k,w}.
	$$
	
	To compute these differences, we use analytic continuation along paths determined by chosen words in $S$ and $T$. Monodromy around the poles is not included in $\Lambda_{k,w}$ and is considered separately in Proposition~\ref{prop:CM-diagonal-monodromy}.
	\begin{theorem}
		\label{thm:exact-period-module}
		The subgroups $\Lambda_{k,w}\subset\mathbb C$ are given by
		$$
		\Lambda_{4,w}=4\pi^2\mathbb Z\widehat\Pi_1^{(4)}(w),\quad
		\Lambda_{6,w}=4\pi^3i\mathbb Z\widehat\Pi_1^{(6)}(w),
		$$
		$$
		\begin{gathered}
			\Lambda_{8,w}=8\pi^4\mathbb Z\widehat\Pi_1^{(8)}(w)+4\pi^4\mathbb Z\left(\widehat\Pi_1^{(8)}(w)+\widehat\Pi_2^{(8)}(w)\right)
			+\frac{8\pi^4}{3}\mathbb Z\widehat\Pi_3^{(8)}(w),
		\end{gathered}
		$$
		$$
		\begin{gathered}
			\Lambda_{10,w}=16\pi^5i\mathbb Z\widehat\Pi_1^{(10)}(w)+8\pi^5i\mathbb Z\widehat\Pi_2^{(10)}(w)
			+\frac{4\pi^5i}{3}\mathbb Z\left(2\widehat\Pi_1^{(10)}(w)+3\widehat\Pi_2^{(10)}(w)+2\widehat\Pi_3^{(10)}(w)\right),
		\end{gathered}
		$$
		$$
		\refstepcounter{equation}\label{eq:Lambda-k}
		\begin{gathered}
			\Lambda_{14,w}=32\pi^7i\mathbb Z\widehat\Pi_1^{(14)}(w)+16\pi^7i\mathbb Z\left(\widehat\Pi_1^{(14)}(w)+\widehat\Pi_2^{(14)}(w)\right)\\
			+\frac{16\pi^7i}{9}\mathbb Z\left(12\widehat\Pi_1^{(14)}(w)+\widehat\Pi_3^{(14)}(w)\right)+\frac{8\pi^7i}{3}\mathbb Z\left(\widehat\Pi_2^{(14)}(w)+\widehat\Pi_4^{(14)}(w)\right)\\
			+\frac{4\pi^7i}{45}\mathbb Z\left(126\widehat\Pi_1^{(14)}(w)+15\widehat\Pi_2^{(14)}(w)+15\widehat\Pi_4^{(14)}(w)+6\widehat\Pi_5^{(14)}(w)\right).
		\end{gathered}\eqno(\theequation)
		$$
	\end{theorem}
	
	\begin{proof}
		Consider the coefficients of the polynomial $P_Q(u)$ in \eqref{eq:I-bullet-T-general}. By \eqref{relation:raised-periods-reflection}, it is enough to work with the following independent raised periods:
		$$
		\widehat\Pi_1^{(4)};\quad
		\widehat\Pi_1^{(6)};\quad
		\widehat\Pi_1^{(8)},\widehat\Pi_2^{(8)},\widehat\Pi_3^{(8)};\quad
		\widehat\Pi_1^{(10)},\widehat\Pi_2^{(10)},\widehat\Pi_3^{(10)};\quad
		\widehat\Pi_1^{(14)},\ldots,\widehat\Pi_5^{(14)}.
		$$
		
		The subgroup generated by these coefficients can be determined from finitely many forms. Each coefficient is an integer polynomial in $a,b,c$ of degree at most $k/2-1$ in each variable. For a polynomial $f$ of degree at most	$k/2-1$, Newton's interpolation formula gives
		$$
		f(n)=\sum_{s=0}^{k/2-1}\binom ns\Delta^sf(0),\qquad\Delta f(n)=f(n+1)-f(n).
		$$
		Each $\Delta^sf(0)$ is an integer linear combination of $f(0),\ldots,f(s)$. Hence all values $f(n)$, $n\in\mathbb Z$, are generated by $f(0),\ldots,f(k/2-1)$. Applying this formula successively to $a,b,c$, we find that it is enough to consider the range $0\le a,b,c\le k/2-1$.
		
		The coefficient vectors for these triples form the rows of integer matrices whose Hermite normal forms are
		$$
		M_4=M_6=(2),\quad
		M_8=\left(\begin{smallmatrix}6&0&0\\3&3&0\\0&0&2\end{smallmatrix}\right),\quad
		M_{10}=\left(\begin{smallmatrix}24&0&0\\0&12&0\\4&6&4\end{smallmatrix}\right),\quad
		M_{14}=\left(\begin{smallmatrix}360&0&0&0&0\\180&180&0&0&0\\240&0&20&0&0\\0&30&0&30&0\\126&15&0&15&6\end{smallmatrix}\right).
		$$
		The same Hermite normal forms are obtained by restricting to positive-definite integral forms with $1\le a,c\le3$, so the corresponding differences generate $\Lambda_{k,w}$.
		
		The common factors in \eqref{eq:I-bullet-T-general} for $k=4,6,8,10,14$ are, respectively,
		$$
		2\pi^2,\qquad-2\pi^3i,\qquad-\frac{4\pi^4}{3},\qquad\frac{2\pi^5i}{3},\qquad-\frac{4\pi^7i}{45}.
		$$
		Multiplying the rows of $M_k$ by these factors gives \eqref{eq:Lambda-k}.
	\end{proof}
	
	\begin{corollary}
		\label{cor:I-bullet-class-invariant}
		For fixed $w$ and $z_Q\notin\Gamma w$, the class
		$$
		\refstepcounter{equation}\label{eq:I-bullet-class}
		\left[\widehat I_{k/2}^\bullet(Q;w)\right]\in\mathbb C/\Lambda_{k,w}
		\eqno(\theequation)
		$$
		depends only on the $\Gamma$-equivalence class of $Q$.
	\end{corollary}
	
	In \cite{Mel08}, Mellit proved that at CM points, after normalization by the appropriate power of the discriminant, the values of his complex lift associated with two equivalent quadratic forms differ by an element of $2\pi i\mathbb Q$. In the present setting, we determine the subgroup formed by these differences explicitly.
	
	\begin{proposition}
		\label{prop:CM-period-module-discrete}
		Let $R(X,Y)=AX^2-BXY+CY^2$ be a positive-definite integral quadratic form of discriminant $D_R<0$, and suppose that
		$$
		w_R=\frac{B+\sqrt{D_R}}{2A}\in\Delta_\sigma,\qquad w_R\notin\Gamma\cdot\{i,\rho\}.
		$$
		Then, for $1\le j\le k-3$,
		$$
		\refstepcounter{equation}\label{eq:CM-raised-period-rational}
		D_R^{(k-2)/4}(2\pi i)^{(k-2)/2}\widehat\Pi_j^{(k)}(w_R)\in\mathbb Q.
		\eqno(\theequation)
		$$
		Consequently, for some $q\in\mathbb Q_{\ge0}$,
		$$
		D_R^{(k-2)/4}\Lambda_{k,w_R}=2\pi iq\mathbb Z.
		$$
	\end{proposition}
	
	\begin{proof}
		On $\Delta_\sigma$, the periods $\Pi_j^{(k)}(w)$, $1\le j\le k-3$, are polynomials in $w$ with rational coefficients. Applying Lemma~\ref{lem:CM-raised-monomials} to these polynomials gives \eqref{eq:CM-raised-period-rational}.
		
		Together with \eqref{eq:Lambda-k}, this shows that
		$\frac{D_R^{(k-2)/4}}{2\pi i}\Lambda_{k,w_R}\subset\mathbb Q
		$ is a finitely generated subgroup. Since every finitely generated subgroup of $\mathbb Q$ is cyclic, it has the form $q\mathbb Z$ for some $q\in\mathbb Q_{\ge0}$.
	\end{proof}
	
	Let $w=w_R$ be a CM point. If $Q$ and $Q'$ are equivalent forms, then
	$$
	D_R^{(k-2)/4}\left(\widehat I_{k/2}^\bullet(Q';w_R)-\widehat I_{k/2}^\bullet(Q;w_R)\right)\in2\pi i\mathbb Q.
	$$
	Thus replacing $Q$ by an equivalent form multiplies
	$$
	\exp\left(D_R^{(k-2)/4}\widehat I_{k/2}^\bullet(Q;w_R)\right)
	$$
	by a root of unity.

	\begin{theorem}
		\label{thm:CM-period-root-unity-general}
		Let $\mu_N$ denote the group of roots of unity of order dividing $N$, and set $N_k=1$ for $k=4,8$ and $N_k=4$ for $k=6,10,14$. For $R$ as in Proposition~\ref{prop:CM-period-module-discrete}, we have
		$$
		\refstepcounter{equation}\label{eq:CM-period-uniform-bound}
		\exp\left(D_R^{(k-2)/4}\Lambda_{k,w_R}\right)\subseteq\mu_{N_k}.
		\eqno(\theequation)
		$$
		If $R$ is primitive, then
		$$
		\refstepcounter{equation}\label{eq:CM-period-root-unity-general}
		\exp\left(D_R^{(k-2)/4}\Lambda_{k,w_R}\right)=
		\begin{cases}
			\{1\},&k=4,8,\\
			\mu_4,&k=6,10,14,\ D_R\equiv1\pmod4,\\
			\mu_2,&k=6,10,\ D_R\equiv0\pmod4,\\
			\{1\},&k=14,\ D_R\equiv0\pmod4.
		\end{cases}
		\eqno(\theequation)
		$$
	\end{theorem}
	
	\begin{proof}
		The calculations in Appendix~\ref{appendix:general-CM-denominators} show that the generators $b_{k,\nu}$ of
		$\frac{D_R^{(k-2)/4}}{2\pi i}\Lambda_{k,w_R}$
		are integers for $k=4,8$ and belong to $\frac14\mathbb Z$ for $k=6,10,14$. This proves \eqref{eq:CM-period-uniform-bound}.
		
		Suppose that $R=[A,B,C]$ is primitive. If $B$ is odd, then $D_R\equiv1\pmod4$. By \eqref{eq:general-CM-congruences}, for each $k=6,10,14$, at least one generator has denominator $4$. Hence the corresponding roots of unity form $\mu_4$.
		
		If $B$ is even, then $D_R\equiv0\pmod4$. The same congruences show that all $b_{14,\nu}$ are integers, so only $1$ occurs for $k=14$. For $k=6,10$, we have $b_{k,\nu}\in\frac12\mathbb Z$. Since $R$ is primitive, at least one of $A,C$ is odd, and hence $A+C+AC\equiv1\pmod2$. By \eqref{eq:general-CM-congruences}, at least one generator is nonintegral for each $k=6,10$. Thus the corresponding roots of unity form $\mu_2$.
	\end{proof}

	In particular, for $w_*=(1+i\sqrt d)/2$, where $d>3$ and $d\equiv3\pmod4$,
	$$
	\exp\left((-d)^{(k-2)/4}\Lambda_{k,w_*}\right)=
	\begin{cases}
		\{1\},&k=4,8,\\
		\mu_4,&k=6,10,14.
	\end{cases}
	$$
	
	\section{Complex CM values of the extended higher Green function}
	\label{sec:extended-gross-zagier}
	
	\subsection{Arithmetic of CM values of the extended Green function}
	\label{subsec:CM-exponential}
	
	As before, let $k\in S$, and let $Q=[a,b,c]$ and $R=[A,B,C]$ be positive-definite integral quadratic forms. Choose a representative $R$ of its $\Gamma$-equivalence class such that $w_R\in\Delta_\sigma$. Assume also that $w_R\notin\Gamma\cdot\{i,\rho\}$ and $z_Q\notin\Gamma w_R$. Fix the branches
	$\sqrt{D_Q}=i\sqrt{|D_Q|}$ and $\sqrt{D_R}=i\sqrt{|D_R|}$.

	\begin{proposition}
		\label{prop:CM-diagonal-monodromy}
		The values of $D_Q^{(k-2)/4}D_R^{(k-2)/4}\mathcal G_{k/2}^\bullet(z_Q,w_R)$ obtained by analytic continuation along different paths differ by an element of $2\pi i\mathbb Z$.
	\end{proposition}
	
	\begin{proof}
		It suffices to compute the monodromy around the points of $\Gamma w_R$. Since $H_k^\bullet-H_k$ is a polynomial in $z$, the functions $H_k^\bullet$ and $H_k$ have the same monodromy. By \eqref{eq:H-diagonal-log},
		$\operatorname{Mon}_{w_R}H_k^\bullet(z,w_R)=C_k(z-w_R)^{k-2}$.
		Applying the Maass raising operators in both variables and evaluating at $z=z_Q$ and $w=w_R$, we obtain from \eqref{eq:raised-CM-kernel-bullet}
		\begin{multline}
			\operatorname{Mon}_{w_R}\left(
			D_Q^{(k-2)/4}D_R^{(k-2)/4}\mathcal G_{k/2}^\bullet(z_Q,w_R)
			\right)
			=\\
			2\pi i\sum_{r=0}^{\lfloor(k-2)/4\rfloor}
			\binom{k/2-1}{2r}\binom{2r}{r}
			(2aC+2cA-bB)^{k/2-1-2r}\\
			\times
			\left((aC-cA)^2+(aB-bA)(cB-bC)\right)^r.
		\end{multline}
		All binomial coefficients and expressions in brackets are integers, so the right-hand side belongs to $2\pi i\mathbb Z$.
		
		The same calculation applies at every point of $\Gamma w_R$. Since these are the only singularities, the monodromy along any closed path lies in $2\pi i\mathbb Z$.
	\end{proof}
	
	\begin{proposition}
		\label{prop:extended-CM-modulus}
		Let $Q'$ be a positive-definite integral quadratic form equivalent to $Q$. Then
		$$
		\refstepcounter{equation}\label{eq:extended-CM-root-of-unity-ambiguity}
		\frac{\exp\left(D_{Q'}^{(k-2)/4}D_R^{(k-2)/4}\mathcal G_{k/2}^\bullet(z_{Q'},w_R)\right)}
		{\exp\left(D_Q^{(k-2)/4}D_R^{(k-2)/4}\mathcal G_{k/2}^\bullet(z_Q,w_R)\right)}
		\in\exp\left(D_R^{(k-2)/4}\Lambda_{k,w_R}\right).
		\eqno(\theequation)
		$$
		In particular, this ratio is a root of unity. In addition,
		$$
		\refstepcounter{equation}\label{eq:extended-CM-modulus}
		\left|\exp\left(D_Q^{(k-2)/4}D_R^{(k-2)/4}\mathcal G_{k/2}^\bullet(z_Q,w_R)\right)\right|^2
		=\exp\left(|D_QD_R|^{(k-2)/4}G_{k/2}(z_Q,w_R)\right).
		\eqno(\theequation)
		$$
	\end{proposition}
	
	\begin{proof}
		By Proposition~\ref{prop:CM-diagonal-monodromy}, the exponentials of the normalized CM values of $\mathcal G_{k/2}^\bullet$ do not depend on the choice of path of analytic continuation. For equivalent forms $Q$ and $Q'$, Corollary~\ref{cor:I-bullet-class-invariant} gives
		$$
		D_{Q'}^{(k-2)/4}\mathcal G_{k/2}^\bullet(z_{Q'},w_R)
		-D_Q^{(k-2)/4}\mathcal G_{k/2}^\bullet(z_Q,w_R)
		\in\Lambda_{k,w_R}.
		$$
		Multiplying by $D_R^{(k-2)/4}$ and taking exponentials, we obtain \eqref{eq:extended-CM-root-of-unity-ambiguity}. By Theorem~\ref{thm:CM-period-root-unity-general}, the right-hand side consists of roots of unity.
		
		The identity \eqref{eq:extended-CM-modulus} follows from \eqref{eq:G-bullet-exact-real-part}.
	\end{proof}
	
	In our setting, the complex Gross--Zagier conjecture proposed by Mellit \cite[\S1.5]{Mel08} takes the following form.
	\begin{conjecture}
		\label{conj:extended-gross-zagier-general}
		Let $k\in S$, and let $Q$ and $R$ be positive-definite integral quadratic forms with discriminants $D_Q,D_R<0$. Assume that
		$w_R\in\Delta_\sigma$, $w_R\notin\Gamma\cdot\{i,\rho\}$, and $z_Q\notin\Gamma w_R$. Then
		$$
		\refstepcounter{equation}\label{eq:extended-gross-zagier-general}
		\exp\left(D_Q^{(k-2)/4}D_R^{(k-2)/4}\mathcal G_{k/2}^\bullet(z_Q,w_R)\right)
		\in\overline{\mathbb Q}^{\times}.
		\eqno(\theequation)
		$$
	\end{conjecture}

	The algebraicity conjecture is unchanged if $H_k^\bullet$ is replaced by $H_k^\bullet+C_kP$, where $P\in\mathbb Q[z,w]$ has degree at most $k-2$ in each variable. Indeed, Lemma~\ref{lem:CM-raised-monomials} shows that the exponentials of the corresponding normalized CM values differ by a root of unity.

	\subsection{Complex conjugation and inverse classes of quadratic forms}
	\label{subsec:CM-conjugation}
	
	Let $w_*=(1+i\sqrt d)/2$, where $d>3$ and $d\equiv3\pmod4$. Since $w_*-1\in\Delta_\sigma$ and the period module is $T$-invariant, we have $\Lambda_{k,w_*}=\Lambda_{k,w_*-1}$.
	
	For a form $Q=[a,b,c]$, set $Q^\vee=[a,-b,c]$. Then
	$z_{Q^\vee}=-\bar z_Q$, and if $Q$ is primitive, then $[Q^\vee]=[Q]^{-1}$.
	For $z_Q\notin\Gamma w_*$, set
	$$
	\refstepcounter{equation}\label{eq:alpha-Q-definition}
	\alpha_Q:=
	\exp\left(
	D_Q^{(k-2)/4}(-d)^{(k-2)/4}
	\mathcal G_{k/2}^\bullet(z_Q,w_*)
	\right).
	\eqno(\theequation)
	$$
	
	\begin{lemma}
		\label{lem:CM-conjugation}
		We have $\alpha_{Q^\vee}=\overline{\alpha_Q}$.
	\end{lemma}
	
	\begin{proof}
		By \eqref{eq:Pi-conjugation} and the definition of $H_k^\bullet$, we have
		$$
		\overline{\mathcal G_{k/2}^\bullet(z,w)}
		=\mathcal G_{k/2}^\bullet(-\bar z,-\bar w).
		$$
		Since $-\bar w_*=w_*-1$ and $\mathcal G_{k/2}^\bullet$ is $T$-invariant in the second variable, we obtain
		$\overline{\mathcal G_{k/2}^\bullet(z,w_*)}=\mathcal G_{k/2}^\bullet(-\bar z,w_*)$.
		Finally, $z_{Q^\vee}=-\bar z_Q$ and $D_Q^{(k-2)/4}(-d)^{(k-2)/4}\in\mathbb R$, hence
		$\alpha_{Q^\vee}=\overline{\alpha_Q}$.
	\end{proof}
	
	\begin{corollary}
		\label{cor:extended-GZ-ambiguous}
		If $Q^\vee$ is equivalent to $Q$, then
		$\alpha_Q\in\overline{\mathbb Q}^{\times}$ and
		$\alpha_Q/|\alpha_Q|\in\mu_{2N_k}$.
		In particular, $\alpha_Q$ is real for $k=4,8$. Thus Conjecture~\ref{conj:extended-gross-zagier-general} holds under the stated assumption, which is satisfied by every primitive class of order $1$ or $2$.
		\end{corollary}
	
	\begin{proof}
		By Lemma~\ref{lem:CM-conjugation}, we have $\alpha_{Q^\vee}=\overline{\alpha_Q}$. Since $Q^\vee$ is equivalent to $Q$, equations \eqref{eq:extended-CM-root-of-unity-ambiguity} and \eqref{eq:CM-period-uniform-bound} imply that
		$\overline{\alpha_Q}/\alpha_Q\in\mu_{N_k}$, and hence
		$\alpha_Q^{2N_k}=|\alpha_Q|^{2N_k}$.
		
		On the other hand, equation \eqref{eq:extended-CM-modulus} and the known logarithmic algebraicity result for CM values of the higher Green function \cite{BEY21,BLY25} give
		$|\alpha_Q|^2\in\overline{\mathbb Q}^{\times}$. Therefore
		$\alpha_Q\in\overline{\mathbb Q}^{\times}$ and
		$\alpha_Q/|\alpha_Q|\in\mu_{2N_k}$.
	\end{proof}
	
	If $Q^\vee$ is not equivalent to $Q$, we still have $\overline{\alpha_Q}=\alpha_{Q^\vee}$. The following example shows that $\alpha_Q$ need not be real.
	
	\begin{example}
		\label{ex:extended-gross-zagier-k4-D23-d7}
		Let $k=4$, $w_*=(1+i\sqrt7)/2$, and
		$Q(X,Y)=2X^2-XY+3Y^2$. Then $D_Q=-23$, $z_Q=(1+i\sqrt{23})/4$, and $D_Q^{1/2}(-7)^{1/2}=-\sqrt{161}$.
		
		By Theorem~\ref{thm:exact-period-module} and the value of
		$\widehat\Pi_1^{(4)}(w_*)$, we have
		$(-7)^{1/2}\Lambda_{4,w_*}=10\pi i\mathbb Z$.
		Thus $\alpha_Q$ depends only on the equivalence class of $Q$. A numerical computation gives
		$$
		\alpha_Q\approx
		5.6541585260369068\cdot10^{-26}
		-2.0486128907112288\cdot10^{-25}i.
		$$
		
		Set
		$$
		\beta:=
		\left(11775+928\sqrt{161}\right)^5
		\left(\frac{(12+\sqrt{161})^2}{17}\right)^{-3}
		\left(\frac{(25+2\sqrt{161})^2}{19}\right)^2\alpha_Q
		$$
		and define
		$$
		\refstepcounter{equation}\label{eq:numerical-complex-cubic-alpha-bullet-D23-d7}
		\begin{aligned}
			P(X):={}&X^3+
			\frac{687414606829246827+54195363611526117\sqrt{161}}
			{277123906250}X^2\\
			&+\frac{1326089276410904661-104627302808963931\sqrt{161}}
			{942221281250}X
			+\frac{99+4\sqrt{161}}{85}.
		\end{aligned}
		\eqno(\theequation)
		$$
		The numerical computation gives $|P(\beta)|<5.1\cdot10^{-143}$.
	\end{example}

	\begin{appendix}
		\section[The zeroth holomorphic period]{The holomorphic period $\Pi_0^{(k)}$}
		\label{appendix:Pi0}
		
		The period $\Pi_0^{(k)}(w)$ is defined by integration along the positive imaginary axis. Throughout this appendix, we assume that $w\in\Delta_\sigma$. The values $H_k(i,w)$ and $\partial_zH_k(i,w)$ are taken on the branch obtained by analytic continuation along the positive imaginary axis.
		
		To compute $\Pi_0^{(k)}(w)$, we use the $S$-transformation formula for $H_k$ at $z=i$. For $k=6,10,14$, the term involving $\Pi_0^{(k)}(w)$ cancels because $i^{k-2}=1$, so we differentiate the formula first. The required values of $H_k(i,w)$ and $\partial_zH_k(i,w)$ are obtained from the following relation between $H_k(z,w)$ and $H_k(w,z)$.
		
		\begin{lemma}
			\label{lemma:H-reciprocity}
			If $\Im z\Im w>1$ and $0<\Re(z-w)<1$, then
			$$
			\refstepcounter{equation}\label{eq:H-reciprocity}
			C_k^{-1}\bigl(H_k(z,w)-H_k(w,z)\bigr)
			=\frac{B_{k-1}(z-w)}{k-1}
			+\frac{2k}{B_kC_k}
			\sum_{n\ge1}\frac{\sigma_{k-1}(n)}{n^{k-1}}
			\bigl(q_w^n-q_z^n\bigr),
			\eqno(\theequation)
			$$
			where $B_n(X)$ denotes the Bernoulli polynomial defined by
			$\frac{te^{Xt}}{e^t-1}=\sum_{n=0}^\infty B_n(X)\frac{t^n}{n!}.$ Here $H_k(z,w)$ and $H_k(w,z)$ are obtained by analytic continuation of their respective Fourier expansions.
		\end{lemma}
		\begin{proof}
			In \eqref{zag:H}, the double sum is symmetric in $z$ and $w$. Hence, after analytic continuation to the region $\Im z\Im w>1$,
			$$
			H_k(z,w)-H_k(w,z)
			=-\operatorname{Li}_{k-1}(q_z/q_w)
			+\operatorname{Li}_{k-1}(q_w/q_z)
			+\frac{2k}{B_k}\sum_{n\ge1}\frac{\sigma_{k-1}(n)}{n^{k-1}}(q_w^n-q_z^n).
			$$
			For $0<\Re(z-w)<1$, the chosen branch satisfies
			$\log(-q_z/q_w)=2\pi i(z-w)-\pi i$. The inversion formula for the polylogarithm \cite{Lewin81},
			$$
			\operatorname{Li}_{k-1}(q_z/q_w)-\operatorname{Li}_{k-1}(q_w/q_z)
			=-\frac{(2\pi i)^{k-1}}{(k-1)!}B_{k-1}(z-w),
			$$
			then gives \eqref{eq:H-reciprocity}.
		\end{proof}
		
		We shall also need the following identity for the Eichler integral of the Eisenstein series at $i$. Set $\varepsilon_k=0$ for $k=4,8$ and $\varepsilon_k=1$ for $k=6,10,14$.
		
		\begin{lemma}
			\label{lemma:Eichler-i-combination}
			For $k\in S$, we have
			$$
			\refstepcounter{equation}\label{eq:Eichler-i-combination}
			\sum_{n\ge1}\frac{\sigma_{k-1}(n)}{n^{k-1}}e^{-2\pi n}
			+\frac{4\pi\varepsilon_k}{k-2}
			\sum_{n\ge1}\frac{\sigma_{k-1}(n)}{n^{k-2}}e^{-2\pi n}
			=r_k\pi^{k-1}-\frac{\zeta(k-1)}2,
			\eqno(\theequation)
			$$
			where
			$$
			\begin{gathered}
				r_4=\frac7{2^3\cdot3^2\cdot5},\quad
				r_6=\frac{13}{2^3\cdot3^3\cdot5\cdot7},\quad
				r_8=\frac{19}{2^3\cdot3^4\cdot5^2\cdot7},\\
				r_{10}=\frac{127}{2^4\cdot3^5\cdot5^2\cdot7\cdot11},\quad
				r_{14}=\frac{3989}{2^2\cdot3^8\cdot5^3\cdot7^2\cdot11\cdot13}.
			\end{gathered}
			$$
		\end{lemma}
		
		\begin{proof}
			We use Ramanujan's identity \cite[(48)]{BerndtStraub16} together with
			$$
			\sum_{n\ge1}\frac{1}{n^{k-1}(e^{2\alpha n}-1)}
			=\sum_{n\ge1}\frac{\sigma_{k-1}(n)}{n^{k-1}}e^{-2\alpha n}.
			$$
			For $k=4,8$, we set $\alpha=\beta=\pi$. For $k=6,10,14$, we substitute $\beta=\pi^2/\alpha$, differentiate with respect to $\alpha$, and evaluate at $\alpha=\pi$. Simplifying these identities gives \eqref{eq:Eichler-i-combination}, where the constants $r_k$ are obtained from the finite sum of Bernoulli numbers in Ramanujan's identity.
		\end{proof}

		Combining these results, we obtain the following formula for $\Pi_0^{(k)}(w)$. Recall that $C_k=(2\pi i)^{k-1}/(k-2)!$.
		
		\begin{proposition}
			\label{thm:Pi0-closed-form}
			For $k\in S$ and $w\in\Delta_\sigma$, we have
			$$
			\refstepcounter{equation}\label{eq:Pi0-closed-form}
			\Pi_0^{(k)}(w)
			=w+\frac12-\frac{B_{k-1}(w+1)}{k-1}
			+\frac{2k}{B_kC_k}
			\left(
			\frac{\zeta(k-1)}2+
			\sum_{n\ge1}\frac{\sigma_{k-1}(n)}{n^{k-1}}q_w^n
			\right).
			\eqno(\theequation)
			$$
			In particular,
			$D_w^{k-1}\Pi_0^{(k)}(w)=-C_k^{-1}E_k(w)$.
		\end{proposition}
		
		\begin{proof}
			Assume that $\Im w>1$, and set $\varepsilon_k=0$ for $k=4,8$ and $\varepsilon_k=1$ for $k=6,10,14$. Evaluating \eqref{eq:S-transformation-hw} at $z=i$ in the first case and its derivative with respect to $z$ at $z=i$ in the second, we obtain from \eqref{eq:Pi-reflection}
			$$
			\refstepcounter{equation}\label{eq:Pi0-at-i-corrected}
			\Pi_0^{(k)}(w)
			=C_k^{-1}\left(
			H_k(i,w)-\frac{2\varepsilon_ki}{k-2}\partial_zH_k(i,w)
			\right)-\beta_k(w),
			\eqno(\theequation)
			$$
			where
			$$
			\beta_k(w)=
			\begin{cases}
				-\dfrac12\displaystyle\sum_{j=1}^{k-3}
				(-1)^j\binom{k-2}{j}i^{k-2-j}\Pi_j^{(k)}(w),
				&k=4,8,\\[6pt]
				\dfrac1{k-2}\displaystyle\sum_{j=1}^{k-3}
				(-1)^j\binom{k-2}{j}(k-2-j)i^{k-2-j}\Pi_j^{(k)}(w),
				&k=6,10,14.
			\end{cases}
			$$
			
			Since $g_{-n}^{(k)}$ has weight $2-k$, its transformation under $S$ at the fixed point $i$ gives
			$$
			g_{-n}^{(k)}(i)-\frac{2\varepsilon_ki}{k-2}(g_{-n}^{(k)})'(i)=0.
			$$
			Consequently, the Fourier expansion of $H_k(w,z)$ yields
			$$
			\left.
			\left(
			1-\frac{2\varepsilon_ki}{k-2}\frac{\partial}{\partial z}
			\right)H_k(w,z)
			\right|_{z=i}=0.
			$$
			
			Applying the same differential operator to \eqref{eq:H-reciprocity} and using Lemma~\ref{lemma:Eichler-i-combination}, we find
			$$
			\begin{aligned}
				\Pi_0^{(k)}(w)
				={}&\frac{B_{k-1}(i-w)}{k-1}
				-\frac{2\varepsilon_ki}{k-2}B_{k-2}(i-w)-\beta_k(w)\\
				&+\frac{2k}{B_kC_k}
				\left(
				\frac{\zeta(k-1)}2+
				\sum_{n\ge1}\frac{\sigma_{k-1}(n)}{n^{k-1}}q_w^n
				-r_k\pi^{k-1}
				\right).
			\end{aligned}
			$$
			The explicit formulas for $\Pi_j^{(k)}(w)$ and $r_k$ give
			$$
			\frac{B_{k-1}(i-w)}{k-1}
			-\frac{2\varepsilon_ki}{k-2}B_{k-2}(i-w)
			-\beta_k(w)
			-\frac{2k(k-2)!r_k}{B_k(2i)^{k-1}}
			=
			w+\frac12-\frac{B_{k-1}(w+1)}{k-1}.
			$$
			Hence \eqref{eq:Pi0-closed-form} holds for $\Im w>1$. Since both sides are holomorphic on $\Delta_\sigma$, the identity extends throughout $\Delta_\sigma$. Applying $D_w^{k-1}$ gives
			$D_w^{k-1}\Pi_0^{(k)}(w)=-C_k^{-1}E_k(w)$.
		\end{proof}

		\section{Eichler periods on $\Delta_\sigma$ and Maass-raised periods at CM points}
		\label{appendix:period-tables}
		
		\subsection{Periods on $\Delta_\sigma$}
		\label{appendix:explicit-periods}
		
		For $w\in\Delta_\sigma\setminus\Gamma\cdot\{i,\rho\}$, comparing coefficients in \eqref{eq:contour-identity-general} gives the following formulas. The cases $k=4,6$ were given in Example~\ref{ex:Pi-closed-forms}.
		
		\noindent\textbf{$k=8$.}
		$$
		\begin{aligned}
			\Pi_1^{(8)}(w)&=\frac1{12}\bigl(2w^6+6w^5+5w^4+5w^2+6w+2\bigr),\\
			\Pi_2^{(8)}(w)&=-\frac1{10}w(w-1)(w+1)(w+2)(2w+1),\\
			\Pi_3^{(8)}(w)&=\frac12w^2(w+1)^2.
		\end{aligned}
		$$
		
		\noindent\textbf{$k=10$.}
		$$
		\begin{aligned}
			\Pi_1^{(10)}(w)&=\frac1{24}(w^2+w+1)
			\bigl(3w^6+9w^5+2w^4-11w^3+2w^2+9w+3\bigr),\\
			\Pi_2^{(10)}(w)&=-\frac1{14}w(w-1)(w+1)(w+2)(2w+1)(w^2+w+1),\\
			\Pi_3^{(10)}(w)&=\frac16w^2(w+1)^2(w^2+w+1).
		\end{aligned}
		$$
		
		\noindent\textbf{$k=14$.}
		$$
		\begin{aligned}
			\Pi_1^{(14)}(w)&=\frac1{24}
			\bigl(2w^{12}+12w^{11}+22w^{10}-33w^8+44w^6-33w^4+22w^2+12w+2\bigr),\\
			\Pi_2^{(14)}(w)&=-\frac1{66}w(w-1)(w+1)(w+2)(2w+1)
			\bigl(3w^6+9w^5+5w^4-5w^3+5w^2+9w+3\bigr),\\
			\Pi_3^{(14)}(w)&=\frac1{20}w^2(w+1)^2
			\bigl(2w^6+6w^5+w^4-8w^3+w^2+6w+2\bigr),\\
			\Pi_4^{(14)}(w)&=-\frac1{18}w^3(w-1)(w+1)^3(w+2)(2w+1),\\
			\Pi_5^{(14)}(w)&=\frac18w^4(w+1)^4.
		\end{aligned}
		$$

		\subsection{Values of the raised periods at the CM point $w_*$}
		\label{appendix:explicit-raised-periods}
		
		Let $d=4t+3$, $t\in\mathbb Z_{>0}$, and $w_*=(1+i\sqrt d)/2$. Substituting $w_*-1=-1/2+i\sqrt d/2$ into the explicit formulas for $\Pi_j^{(k)}$ and applying the Maass raising operator, we obtain the following nonzero independent raised periods.
		
		\noindent\textbf{$k=4$.}
		$$
		\widehat\Pi_1^{(4)}(w_*)
		=\frac{2t+3}{2\pi\sqrt d}.
		$$
		
		\noindent\textbf{$k=6$.}
		$$
		\widehat\Pi_1^{(6)}(w_*)
		=\frac{3(2t^2+4t+3)}{8\pi^2d}.
		$$
		
		\noindent\textbf{$k=8$.}
		$$
		\widehat\Pi_1^{(8)}(w_*)
		=\frac{5(t+1)(t^2+t+1)}{2\pi^3d^{3/2}},
		\qquad
		\widehat\Pi_3^{(8)}(w_*)
		=\frac{3(2t^2+2t+1)}{4\pi^3d^{3/2}}.
		$$
		
		For weights $10$ and $14$, set
		$$
		\begin{aligned}
			p_{10,1}(t)&=30t^4+80t^3+84t^2+56t+21,\\
			p_{10,3}(t)&=5t^3+12t^2+8t+3,\\
			p_{14,1}(t)&=42t^6+168t^5+294t^4+272t^3+150t^2+60t+15,\\
			p_{14,3}(t)&=12t^5+30t^4+40t^3+30t^2+12t+3,\\
			p_{14,5}(t)&=14t^4+32t^3+30t^2+12t+3.
		\end{aligned}
		$$
		
		\noindent\textbf{$k=10$.}
		$$
		\widehat\Pi_1^{(10)}(w_*)
		=\frac{7p_{10,1}(t)}{16\pi^4d^2},
		\qquad
		\widehat\Pi_3^{(10)}(w_*)
		=\frac{p_{10,3}(t)}{2\pi^4d^2}.
		$$
		
		\noindent\textbf{$k=14$.}
		$$
		\widehat\Pi_1^{(14)}(w_*)
		=\frac{165p_{14,1}(t)}{8\pi^6d^3},
		\qquad
		\widehat\Pi_3^{(14)}(w_*)
		=\frac{63p_{14,3}(t)}{8\pi^6d^3},
		\qquad
		\widehat\Pi_5^{(14)}(w_*)
		=\frac{45p_{14,5}(t)}{16\pi^6d^3}.
		$$

		\subsection{Denominators of generators of the normalized period group}
		\label{appendix:general-CM-denominators}
		
		Let $R=[A,B,C]$, $D_R<0$, and $w_R=(B+\sqrt{D_R})/(2A)$. Now we find explicit rational generators of the group
		$\frac{D_R^{(k-2)/4}}{2\pi i}\Lambda_{k,w_R}$.
		\begin{lemma}
			\label{lem:CM-raised-monomials}
			For $m\ge0$ and $0\le r\le2m$,
			$$
			\refstepcounter{equation}\label{eq:CM-raised-monomials}
			D_R^{m/2}(2\pi i)^m
			\left.\mathcal R^{-2m}_{[m],w}w^r\right|_{w=w_R}
			=
			(-1)^{m+r}\frac{r!}{m!}
			\left.
			\frac{d^{2m-r}}{du^{2m-r}}R(u,1)^m
			\right|_{u=0}.
			\eqno(\theequation)
			$$
			The right-hand side belongs to $\mathbb Z[A,B,C]$.
		\end{lemma}
		
		\begin{proof}
			For $k=2m+2$ and $Q=R$, comparison of the coefficients of $u^{2m-r}$ in \eqref{eq:raised-CM-kernel-bullet} gives
			$$
			D_R^{m/2}(2\pi i)^m
			\left.\mathcal R^{-2m}_{[m],w}w^r\right|_{w=w_R}
			=
			(-1)^{m+r}\frac{r!(2m-r)!}{m!}
			\operatorname{coeff}_{u^{2m-r}}R(u,1)^m.
			$$
			Since the coefficient of $u^{2m-r}$ is
			$\frac1{(2m-r)!}\left.\frac{d^{2m-r}}{du^{2m-r}}R(u,1)^m\right|_{u=0}$,
			this proves \eqref{eq:CM-raised-monomials}.
			
			The right-hand side of \eqref{eq:CM-raised-monomials} can be written as
			$$
			(-1)^m
			\sum_{\substack{i+j+\ell=m\\2\ell+j=r}}
			\frac{r!(2m-r)!}{i!j!\ell!}A^iB^jC^\ell.
			$$
			Since $r=2\ell+j$ and $2m-r=2i+j$, we have
			$$
			\frac{r!(2m-r)!}{i!j!\ell!}
			=
			\frac{(2\ell+j)!(2i+j)!}{i!j!\ell!}
			=
			\binom{2i+j}{i}\frac{(i+j)!}{j!}\frac{(2\ell+j)!}{\ell!}
			\in\mathbb Z.
			$$
		\end{proof}
		
		By Theorem~\ref{thm:exact-period-module}, the rows of $M_k$ specify generators of $\Lambda_{k,w_R}$ as linear combinations of the independent raised periods. Let $b_{k,\nu}$ denote these generators multiplied by $D_R^{(k-2)/4}/(2\pi i)$. If $(M_k)_{\nu j}$ denotes the $(\nu,j)$-entry of $M_k$, then
		$$
		\refstepcounter{equation}\label{eq:general-CM-generators}
		b_{k,\nu}
		=
		\frac{(-1)^{(k-2)/2}}{2\left(\frac{k-2}{2}\right)!}
		D_R^{(k-2)/4}(2\pi i)^{(k-2)/2}
		\sum_j(M_k)_{\nu j}\widehat\Pi_j^{(k)}(w_R).
		\eqno(\theequation)
		$$
		Applying Lemma~\ref{lem:CM-raised-monomials} to the explicit formulas for $\Pi_j^{(k)}$, we find that $b_{k,\nu}$ is a polynomial in $A,B,C$ with rational coefficients whose denominators divide $4$. For example,
		$$
		b_{4,1}=2A+B+2C,\qquad
		b_{6,1}=\frac34(2A^2+2AB+2AC+B^2+2BC+2C^2).
		$$
		
		Set $\kappa=3B^2+2(A+C+AB+AC+BC)$. For integral $A,B,C$, the following congruences hold:
		$$
		\refstepcounter{equation}\label{eq:general-CM-congruences}
		\begin{gathered}
			4b_{4,1}\equiv0,\qquad
			4b_{6,1}\equiv\kappa,\qquad
			4b_{8,\nu}\equiv0\quad(1\le\nu\le3),\\
			4b_{10,1}\equiv2B^2,\qquad
			4b_{10,2}\equiv0,\qquad
			4b_{10,3}\equiv\kappa,\\
			4b_{14,\nu}\equiv0\quad(1\le\nu\le4),\qquad
			4b_{14,5}\equiv3B^2
			\pmod4.
		\end{gathered}
		\eqno(\theequation)
		$$
		
		The rational number $q$ in
		Proposition~\ref{prop:CM-period-module-discrete} is
		$q=\frac14\gcd_\nu |4b_{k,\nu}|$.
		
		\begin{example}
			\label{ex:CM-period-generators-d7}
			Let $w_*=(1+i\sqrt7)/2$. For each $k\in S$, the normalized period
			group is
			$(-7)^{(k-2)/4}\Lambda_{k,w_*}=2\pi i q_k\mathbb Z$, where
			$$
			(q_4,q_6,q_8,q_{10},q_{14})
			=\left(5,\frac{27}{4},5,\frac74,\frac{21}{4}\right).
			$$
			These follow from Theorem~\ref{thm:exact-period-module} and
			Appendix~\ref{appendix:explicit-raised-periods} at $t=1$.
		\end{example}
		\section{Expansion of $\widehat I^\bullet_{k/2}(Q;w)$ in terms of raised periods}
		\label{appendix:CM-normalization}
		
		Let $Q=[a,b,c]$ and $z_Q\notin\Gamma w$. Substituting \eqref{eq:S-transformation-hw} into \eqref{eq:H-bullet-definition}, then applying the Maass raising operators and using
		\eqref{relation:raised-periods-reflection}, gives the following explicit
		expansion of $\widehat I^\bullet_{k/2}(Q;w)$.
		
		\begin{corollary}
			\label{cor:I-bullet-explicit}
			Let $Q=[a,b,c]$ and $z_Q\notin\Gamma w$. Then, for $k\in S$, the following expansions hold.
			
			\noindent\textbf{$k=4$.}
			$$
			\refstepcounter{equation}\label{eq:I-bullet-k4}
			\widehat I_2^\bullet(Q;w)
			=
			D_Q^{1/2}\mathcal G_2(z_Q,w)
			-4\pi^2a\widehat\Pi_0^{(4)}(w)
			-2\pi^2b\widehat\Pi_1^{(4)}(w).
			\eqno(\theequation)
			$$
			
			\noindent\textbf{$k=6$.}
			$$
			\refstepcounter{equation}\label{eq:I-bullet-k6}
			\widehat I_3^\bullet(Q;w)
			=D_Q\mathcal G_3(z_Q,w)
			+4\pi^3i\bigl(
			a^2\widehat\Pi_0^{(6)}(w)
			+b(a+c)\widehat\Pi_1^{(6)}(w)
			\bigr).
			\eqno(\theequation)
			$$
			
			\noindent\textbf{$k=8$.}
			$$
			\refstepcounter{equation}\label{eq:I-bullet-k8}
			\begin{aligned}
				\widehat I_4^\bullet(Q;w)
				={}&D_Q^{3/2}\mathcal G_4(z_Q,w)
				+\frac{4\pi^4}{3}\Bigl(
				2a^3\widehat\Pi_0^{(8)}(w)
				+3b(a^2+c^2)\widehat\Pi_1^{(8)}(w)\\
				&\qquad
				+3(a-c)(ac+b^2)\widehat\Pi_2^{(8)}(w)
				+b(6ac+b^2)\widehat\Pi_3^{(8)}(w)
				\Bigr).
			\end{aligned}
			\eqno(\theequation)
			$$
			
			\noindent\textbf{$k=10$.}
			$$
			\refstepcounter{equation}\label{eq:I-bullet-k10}
			\begin{aligned}
				\widehat I_5^\bullet(Q;w)
				={}&D_Q^2\mathcal G_5(z_Q,w)
				-\frac{4\pi^5i}{3}\Bigl(
				a^4\widehat\Pi_0^{(10)}(w)
				+2b(a^3+c^3)\widehat\Pi_1^{(10)}(w)\\
				&\qquad
				+(a^2-c^2)(2ac+3b^2)\widehat\Pi_2^{(10)}(w)
				+2b(a+c)(3ac+b^2)\widehat\Pi_3^{(10)}(w)
				\Bigr).
			\end{aligned}
			\eqno(\theequation)
			$$
			
			\noindent\textbf{$k=14$.}
			$$
			\refstepcounter{equation}\label{eq:I-bullet-k14}
			\begin{aligned}
				\widehat I_7^\bullet(Q;w)
				={}&D_Q^3\mathcal G_7(z_Q,w)
				+\frac{4\pi^7i}{45}\Bigl(
				2a^6\widehat\Pi_0^{(14)}(w)
				+6b(a^5+c^5)\widehat\Pi_1^{(14)}(w)\\
				&\qquad
				+3(a^4-c^4)(2ac+5b^2)\widehat\Pi_2^{(14)}(w)
				+10b(a^3+c^3)(3ac+2b^2)\widehat\Pi_3^{(14)}(w)\\
				&\qquad
				+15(a^2-c^2)(a^2c^2+4ab^2c+b^4)\widehat\Pi_4^{(14)}(w)\\
				&\qquad
				+6b(a+c)(10a^2c^2+10ab^2c+b^4)\widehat\Pi_5^{(14)}(w)
				\Bigr).
			\end{aligned}
			\eqno(\theequation)
			$$
		\end{corollary}
	\end{appendix}

\end{document}